\documentclass[11pt, notitlepage]{article}
\usepackage{amssymb,amsmath,comment}

\usepackage{url}
\catcode`\@=11 \@addtoreset{equation}{section}
\def\thesection{\arabic{section}}

\def\theequation{\thesection.\arabic{equation}}
\catcode`\@=12
\usepackage{colortbl}%
\usepackage{a4wide}
\usepackage{parskip}
\usepackage[hidelinks]{hyperref}
\newcommand{\el}{\ell}
\newcommand{\tl}{\tilde}

\newcommand{\ds} {\displaystyle}
\newcommand{\e}{\epsilon}
\newcommand{\vep}{\varepsilon}
\newcommand{\vth}{\vartheta}
\newcommand{\pa} {\partial}
\newcommand{\al} {\alpha}
\newcommand{\ba} {\beta}
\newcommand{\sg} {\sigma}
\newcommand{\de} {\delta}
\newcommand{\ga} {\gamma}
\newcommand{\Ga} {\Gamma}
\newcommand{\Om} {\Omega}
\newcommand{\ra} {\rightarrow}
\newcommand{\ov}{\overline}
\newcommand{\De} {\Delta}
\newcommand{\la} {\lambda}

\newcommand{\noi} {\noindent}
\newcommand{\na} {\nabla}

\newcommand{\mb} {\mathbb}
\newcommand{\mc} {\mathcal}

\usepackage[all]{xy}
\catcode`\@=11
\def\theequation{\@arabic{\c@section}.\@arabic{\c@equation}}
\catcode`\@=12

\def\QED{\hfill {$\square$}\goodbreak \medskip}

\newtheorem{Theorem}{Theorem}[section]
\newtheorem{Lemma}{Lemma}[section]
\newtheorem{Proposition}{Proposition}[section]
\newtheorem{Corollary}[Theorem]{Corollary}
\newtheorem{Remark}{Remark}[section]
\newtheorem{Definition}{Definition}[section]

\def\Xint#1{\mathchoice
	{\XXint\displaystyle\textstyle{#1}}%
	{\XXint\textstyle\scriptstyle{#1}}%
	{\XXint\scriptstyle\scriptscriptstyle{#1}}%
	{\XXint\scriptscriptstyle\scriptscriptstyle{#1}}%
	\!\int}
\def\XXint#1#2#3{{\setbox0=\hbox{$#1{#2#3}{\int}$ }
		\vcenter{\hbox{$#2#3$ }}\kern-.6\wd0}}

\begin{document}
{\vspace{0.01in}
	\title
	{Global regularity results for non-homogeneous growth fractional problems}
	
	\author{  Jacques Giacomoni$^{\,1}$ \footnote{e-mail: {\tt jacques.giacomoni@univ-pau.fr}}, \ Deepak Kumar$^{\,2}$\footnote{e-mail: {\tt deepak.kr0894@gmail.com}},  \
		and \  Konijeti Sreenadh$^{\,2}$\footnote{
			e-mail: {\tt sreenadh@maths.iitd.ac.in}} \\
		\\ $^1\,${\small Universit\'e  de Pau et des Pays de l'Adour, LMAP (UMR E2S-UPPA CNRS 5142) }\\ {\small Bat. IPRA, Avenue de l'Universit\'e F-64013 Pau, France}\\  
		$^2\,${\small Department of Mathematics, Indian Institute of Technology Delhi,}\\
		{\small	Hauz Khaz, New Delhi-110016, India} }

	\date{}
	
	\maketitle

\begin{abstract}
 The main goal of this article is to show the global H\"older regularity of weak solutions to a class of problems involving the fractional $(p,q)$-Laplacian, denoted by $(-\Delta)^{s_1}_{p}+(-\Delta)^{s_2}_{q}$, for $1<p,q<\infty$ and $s_1,s_2\in (0,1)$. 
We use a suitable Caccioppoli inequality and a local boundedness result in order to prove the weak Harnack inequality. Consequently, by employing a suitable iteration process, we establish the interior H\"older continuity result for local weak solutions. The global H\"older regularity result we prove expands and improves the regularity results of Giacomoni, Kumar and Sreenadh (arXiv: 2102.06080)
to the subquadratic case (that is, $q<2$) and to the more general right hand side, which requires a different and new approach. Moreover,  we establish a nonlocal Harnack type inequality for weak solutions and a strong maximum principle for weak super-solutions, which are of independent interest.
 	\medskip
	
	\noi \textbf{Key words:} Fractional $(p,q)$-Laplacian, non-homogeneous nonlocal operator, local boundedness result, weak Harnack inequality,  interior and boundary H\"older continuity, Harnack inequality, strong maximum principle.
	
	\medskip
	
	\noi \textit{2010 Mathematics Subject Classification:} 35J60, 35R11, 35B45, 35D30.

\end{abstract}

\section{Introduction}
 In this paper, we establish the H\"older continuity results for the weak solutions to the following nonlocal problem:
 \begin{align}\label{prob}
 	(-\Delta)^{s_1}_{p}u+  (-\Delta)^{s_2}_{q}u = f \quad \text{in } \; \Om, 
 \end{align}
 \noi where $\Om$ is a bounded domain in $\mb R^N$ with $C^{1,1}$ boundary, $1< q, p<\infty$, $0<s_2\leq s_1<1$ and $f\in L^\ga_{\rm loc}(\Om)$ with $\ga\begin{cases}
 	> N/(ps_1) \quad\mbox{if }N>ps_1,\\
 	\geq 1 \quad\mbox{otherwise}.
 \end{cases}$ \\
The fractional $p$-Laplacian $(-\Delta)^{s}_{p}$ is defined as
	\begin{equation*}
		{(-\Delta)^{s}_pu(x)}= 2\lim_{\vep\to 0}\int_{\mathbb R^N\setminus B_{\vep}(x)} \frac{|u(x)-u(y)|^{p-2}(u(x)-u(y))}{|x-y|^{N+ps}}dy.
	\end{equation*} 
 These kinds of nonlocal operators have their applications in real world problems, such as obstacle problems, the study of American options in finance, game theory, image processing and anomalous diffusion phenomena (see \cite{nezzaH} for more details). Due to this reason, elliptic problems involving the fractional Laplacian have been extensively studied in the last two decades.\par
 The leading operator, $(-\Delta)^{s_1}_{p}+(-\Delta)^{s_2}_{q}$, in problem \eqref{prob} is known as the fractional $(p,q)$-Laplacian, which is non-homogeneous {(unless $p=q$)} in the sense that for any $t>0$, there does not exist any $\sg\in\mb R$ such that $((-\Delta)^{s_1}_{p}+(-\Delta)^{s_2}_{q})(tu)=t^\sg ((-\Delta)^{s_1}_{p}u+(-\Delta)^{s_2}_{q}u)$ holds for all $u\in W^{s_1,p}(\Om)\cap W^{s_2,q}(\Om)$.
 The operator is the fractional analogue of the $(p,q)$-Laplacian ($-\De_p -\De_q$), which arises from the study of general reaction-diffusion equations with non-homogeneous diffusion and transport aspects. The problems involving these kinds of operators have applications in biophysics, plasma physics and chemical reactions, with double phase features, where the function $u$ corresponds to the concentration term, and the differential operator represents the diffusion coefficient,  for details, see \cite{marano} and  references therein.\par 
 Regarding the regularity results for weak solutions to quasilinear elliptic equations, Lieberman in \cite{lieberm} proved $C^{1,\al}(\ov\Om)$ regularity results for problems containing more general operators than the classical $p$-Laplacian. Subsequently, in \cite{liebm91}, the same author established the interior H\"older continuity of the gradient of weak solutions to problems where the leading operator has Orlicz type of growth condition. See \cite{maly} for more details on the regularity theory of quasilinear problems.  
  \par 
 Concerning the regularity results for problems involving nonlocal operators, the case of the fractional Laplacian is well understood. In particular, in \cite{caffsilv} Caffarelli and Silvestre obtained the interior regularity results, while Ros-Oton and Serra \cite{ros} proved the optimal boundary regularity. Precisely, in the latter work, the authors proved that weak solutions of the fractional Laplacian problem (i.e., $p=q=2$ and $s_1=s_2$) with a bounded right hand side and homogeneous Dirichlet boundary conditions, are in $ C^{0,s}(\mb R^N)$.  One can rely on \cite{dyda} for regularity results of linear fractional problems involving a more general kernel. For the nonlinear and homogeneous operator case ($p=q\neq 2$ and $s_1=s_2$), in \cite{dicastro}, Di Castro et al. established interior H\"older regularity results for nonlocal equations whose prototypes include the fractional $p$-Laplacian. Here, they  established the Caccioppoli inequality and the Logarithmic lemma, and following De Giorgi's iteration technique, proved the local H\"older continuity result.  In \cite{korvenpaa}, Korvenp\"a\"a et al. further obtained the interior and boundary H\"older regularity results for obstacle problems (also covering the Dirichlet problem) with $zero$ right hand side.
 We also mention the work of M. Cozzi \cite{cozzi}, for H\"older regularity results for minimizers of functional involving the fractional $p$-Laplacian energy in the fractional De Giorgi's class.
  Using a slightly different approach, Iannizzotto et al. in \cite{iann}, proved that the weak solutions of problem \eqref{prob} with bounded right hand side, belong to the space $C^{0,\al}(\ov\Om)$, for some $\al\in (0,s_1]$. By constructing the barrier functions as a solution to some nonlocal equation on $B_1$ (fractional torsion problem) and scaling it appropriately to get a sub-solution, they established a weak Harnack type inequality to prove the interior regularity results. Moreover, using the barrier arguments, in the spirit of Krylov's approach (as in \cite{ros}), they obtained the boundary behavior of the weak solutions. 
 Subsequently, in \cite{brascoH}, Brasco et al. established the optimal H\"older interior regularity result and proved that the local weak solution $u\in W^{s,p}_{\rm loc}(\Om)$ of problem \eqref{prob} for the superquadratic and homogeneous case (i.e., $2\leq p=q$ and $s_1=s_2=s$)  with $f\in L^r_{\rm loc}(\Om)$, for suitable $r>0$, is in $C^{0,\al}_{\rm loc}(\Om)$, for all $\al<\min\{1,ps/(p-1)\}$. 
 Consequently, with the help of the boundary behavior of the solution from \cite{iann}, we get the $C^{s}$ regularity result up to the boundary. We additionally refer to \cite{kuusi}, for regularity results for nonlocal problems involving measure data. In \cite{ianndist}, Iannizzotto et al. extended the results of \cite{ros} to the nonlinear setting and proved that the weak solution $u$ of problem \eqref{prob}, again for the case $2\leq p=q$, $s_1=s_2=s\in (0,1)$ and bounded right hand side, satisfies $\frac{u}{d^s} \in C^{0,\al}(\ov\Om)$, for some $\al\in (0,1)$, where $d(x):={\rm dist}(x,\pa\Om)$.  \par 
 Due to the non-homogeneous and nonlinear nature of the fractional $(p,q)$-operators, the corresponding nonlocal problems have attracted the attention of many researchers in recent years. For instance, see {\cite{alves,ambrosio1,ambrosio,ambrosioT1,ambrosioT2,bhakta,DDS}} for the existence and multiplicity results for problems involving the fractional $(p,q)$-Laplacian, and we refer \cite{mingioneJMAA} for a survey of recent developments on non-standard growth problems. On the other hand, there is not much literature available regarding the regularity results. Particularly, in \cite{DDS}, Goel et al. have  obtained the $C^{0,\al}_{\rm loc}$ regularity result, with some unspecified $\al\in (0,1)$, for weak bounded solutions in the superquadratic case.  Further, we mention the recent work of \cite{bonder}, where the global H\"older continuity results (in the spirit of \cite{iann}) for weak solutions to problems involving the fractional $(-\De)_g^s$-Laplacian, where $g$ is a convex Young's function, is proved. However, this does not include our problem, even for the case $s_1=s_2$, because of the power type of growth conditions (2.2) and (2.4), there.  In \cite{filippis}, the authors have obtained the interior H\"older regularity of viscosity solutions to a class of fractional double phase problems. Recently, in \cite{JDS2}, the authors of the present work have obtained almost optimal global H\"older continuity results for weak solutions to fractional $(p,q)$-problems in the superquadratic case ($p\geq q\geq 2$). To establish the optimal interior regularity result for local weak solutions (without the boundedness assumption), the authors proved a local boundedness result, which uses a new Caccioppoli type inequality for non-homogeneous operators. Then, employing Moser's iteration technique on the discrete differential of the solution and exploiting the local boundedness of the function $f$, they proved  suitable Besov space inclusion (which in turn gives H\"older continuity). However, the approach fails when either of the exponent is less than $2$. By establishing the control over the barrier functions involving the distance function, the authors further obtained almost optimal $d^{s_1}$ boundary behavior of the weak solution. This coupled with the interior H\"older regularity result proves the almost optimal $s_1$-H\"older continuity result globally in $\mb R^N$. 
 \par 
 Concerning the Harnack type inequality, in \cite{trudinger}, Trudinger established Harnack and weak Harnack type estimates for weak solutions to  general second order quasilinear problems. As an application to these, the author obtains H\"older continuity for weak solutions. 
 However, in the nonlocal case, the classical Harnack inequality fails, see for instance \cite{Kassmann2}. Subsequently,
  in \cite{Kassmann}, the same author proved the Harnack inequality (where the nonnegativity in the whole $\mb R^N$ is not assumed) with nonlocal tails for problems involving the fractional Laplacian. For general $p$, we refer to \cite{dicastroHarn} for Harnack and weak Harnack type inequality for minimizers as well as weak solutions. In this paper, the authors have used suitable Caccioppoli type estimates and a local boundedness result to achieve these aims.\par
 Inspired from the above discussion, in this work, we answer the open question of H\"older regularity results for weak solutions to fractional $(p,q)$-problems, for the case $1<q< 2$ and $1<p<\infty$. We prove the interior regularity result for local weak solutions, as in Theorem \ref{intreg}, by establishing a weak Harnack type inequality (see Proposition \ref{weakharn}) and then we complete the proof of H\"older continuity in the spirit of \cite[Theorem 2.10]{DDS}. As noted earlier, for the homogeneous fractional $p$-Laplacian case with $p<2$, to prove the weak Harnack inequality, in \cite{iann} a suitably scaled version of the solution to the fractional torsion problem is used for appropriate sub-solutions. Nevertheless, due to the lack of the scaling property, we can not rely on this method for our case. Moreover, we observe that by establishing the Logarithmic lemma (in accordance with \cite[Lemma 1.3]{dicastro} and Lemma \ref{lem1}, below), we can not prove the H\"older continuity result by applying De Giorgi's iteration technique as in the proof of \cite[Theorem 1.2]{dicastro}. This is due to the presence of non-homogeneous power of the parameter $\la$ (taken in the place of $d$, there), which prevents from choosing the suitable constant $k$ involved in the proof of \cite[Lemma 5.1]{dicastro}. So, we follow the idea of \cite{dicastroHarn} to prove our weak Harnack inequality. However, because of the non-homogeneous nature of the operator, we obtain two nonlocal tails with different behaviors, which require a technical care at several places. We overcome the difficulty raised by these nonlocal tails by further improving the parameters involved in the proofs (by adding an additional quantity of the form $r^{s_1p-s_2q}$). Our interior regularity result complements that of \cite{dicastro} to the case of non-homogeneous fractional $(p,q)$-problems with non-homogeneous right hand side in $L^\ga_{\rm loc}(\Om)$, whereas in \cite{JDS2} in addition to $q\geq 2$, $f\in L^\infty_{\rm loc}(\Om)$ is assumed. We strongly believe that this approach, while currently focused on a specific form of the operator, can be further applied for a wider class of nonlocal and non-homogeneous operators, see Remarks \ref{remgenker} and \ref{remmixed} in this regard. 
 It is also worth noting that we do not assume any boundedness condition on the solution in our interior regularity result, thanks to Proposition \ref{localbdd}, which extends that of \cite[Proposition 3.2]{JDS2}. The proof of Proposition \ref{localbdd} relies on a Caccioppoli type inequality for non-homogeneous operators (see \cite[Lemma 3.1]{JDS2}) and De Giorgi's type iteration argument. The main obstacle in the proof is due to $f\in L^\ga_{\rm loc}(\Om)$, for $\ga<\infty$, as we can not use the Moser type iteration argument available for the homogeneous case in \cite{brasco2}. In this paper, clever convexity arguments proved that $u_+$ is a sub-solution with right hand side $|f|$, whenever $u$ is a solution to \eqref{prob} (with $p=q$ and $s_1=s_2$), but seem to be inefficient in presence of non-homogeneous operators. Subsequently, using the boundary behavior of \cite[Proposition 3.11]{JDS2}, we obtain the global H\"older continuity result for weak solutions. Additionally, employing the approach of \cite[Theorem 1.1]{dicastroHarn} to our case, we obtain the Harnack inequality, as in Theorem \ref{harnack}. Similar to the homogeneous case, our Harnack inequality is valid for sign changing solutions also and the sign changing behavior is displayed in terms of the nonlocal tails. As an application to our global H\"older regularity, we obtain a strong maximum principle for fractional $(p,q)$-problems, valid for all $p\geq q>1$. We observe that the proof of \cite[Theorem 2.6]{JDS2} can not be generalized to this case, as our weak Harnack type inequality involves an additional constant term ($r^\frac{ps_1-qs_2}{p-q}$). So, we follow the approach of \cite{delpezzo} and prove that continuous weak super-solutions are viscosity super-solutions of fractional $(p,q)$-problem with right hand side $zero$. Subsequently, we apply a suitable barrier function and the weak comparison principle to extend this result for problems involving more general nonlinearities (see Theorem \ref{strngmax}).
 
  \section{Function Spaces and Main Results}
  We first fix some notations which will be used through out the paper.
  We set  $t_\pm=\max\{\pm t,0\}$.
  We denote $[t]^{p-1}:=|t|^{p-2}t$, for all $p>1$ and $t\in\mb R$.
  Next, for $x_0\in\mb R^N$ and $v\in L^1(B_r(x_0))$, we set \[(v)_{B_r(x_0)}:=\Xint-_{B_r(x_0)}v(x)dx=\frac{1}{|B_r(x_0)|}\int_{B_r(x_0)}v(x)dx.\]
  The order pair $(\el,s)$ should always be considered as $(\el,s)\in \{(p,s_1), (q,s_2)\}$, unless otherwise mentioned.
   The constants $c$ and $C$ may vary line to line.\par
 For any $E\subset\mathbb{R}^N$, $1\leq p<\infty$ and $0<s<1$, the fractional Sobolev space $W^{s,p}(E)$ is defined as 
 \begin{align*}
 	W^{s,p}(E):= \left\lbrace u \in L^p(E): [u]_{W^{s,p}(E)}  < \infty \right\rbrace
 \end{align*}
 \noi endowed with the norm $\|u\|_{W^{s,p}(E)}:=  \|u\|_{L^p(E)}+ [u]_{W^{s,p}(E)}$,
 where \begin{align*}
 	[u]_{W^{s,p}(E)}:=  \left(\int_{E}\int_{E} \frac{|u(x)-u(y)|^p}{|x-y|^{N+sp}}~dxdy  \right)^{1/p}.
 \end{align*} 
 Next, for $1<p<\infty$ and for any (proper) open and bounded subset $E$ of $\mb R^N$, we have 
 \begin{align*}
 	W^{s,p}_0(E):=\{ u\in W^{s,p}(\mb R^N) \ : \ u=0 \quad\mbox{a.e. in }\mb R^N\setminus E \}
 \end{align*}
 which is a uniformly convex Banach space when equipped with the norm  $[\cdot]_{W^{s,p}(\mb R^N)}$ (equivalent to $\| \cdot \|_{W^{s,p}(\mb R^N)}$, and we will denote it by $\|\cdot \|_{W^{s,p}_0(E)}$). 
 When the domain $E$ has Lipschitz boundary, $W^{s,p}_0(E)$ coincides with $X_{p,s}$ (as defined in \cite{DDS}), and the following inclusion holds: 
 \begin{Lemma}\cite[Lemma 2.1]{DDS}
 	Let $E\subset\mb R^N$ be a bounded domain with Lipschitz boundary. Let $1<q\leq p<\infty$ and $0<s_2<s_1 <1$, then there exists a positive constant  $C=C(|E|,\;N,\; p,\;q,\;s_1,\;s_2)$ such that 
 	\begin{align*}
 		\|u\|_{W^{s_2,q}_0(E)}\leq C \|u\|_{W^{s_1,p}_0(E)}, \quad \text{for all } \; u \in W^{s_1,p}_0(E).
 	\end{align*}
 \end{Lemma}
 For any bounded open set $E\subset\mb R^N$ and $0<s_2<s_1<1$, an easy consequence of the H\"older inequality yields 
 \[ \|u\|_{W^{s_2,q}(E)}\leq C \|u\|_{W^{s_1,p}(E)}, \quad \text{for all } \; u \in W^{s_1,p}(E). \]
 In what follows, we will focus only on the case $s_1\neq s_2$ and work with the space $W^{s_1,p}$. The case $s_1=s_2=s$, can be treated similarly by considering the space $\mc W:= W^{s,p}(E)\cap W^{s,q}(E)$ (in place of $W^{s_1,p}(E)$),
 equipped with the norm $\|\cdot\|_{W^{s,p}(E)}+\|\cdot\|_{W^{s,q}(E)}$. 
 \begin{Definition}
 	Let $u:\mb R^N\to \mb R$ be a measurable function, $0< m<\infty$ and $\theta>0$. We define the tail space as below:
 	\begin{align*}
 		L^{m}_{\theta}(\mb R^N) = \bigg\{ u\in L^{m}_{\rm loc}(\mb R^N) : \int_{\mb R^N} \frac{|u(x)|^{m}dx}{(1+|x|)^{N+\theta}} <\infty \bigg\}.
 	\end{align*} 
 	The nonlocal tail of radius $R$ and center $x_0\in\mb R^N$ is defined as 
 	\begin{align*}
 		T_{m,\theta}(u;x_0,R)=\left(R^{\theta}\int_{B_R(x_0)^c} \frac{|u(y)|^{m}}{|x_0-y|^{N+\theta}} dy\right)^{1/m}.
 	\end{align*}
For brevity,  we set $T_{p-1}(u;x,R):=T_{p-1,s_1p}(u;x,R)$ and  $T_{q-1}(u;x,R):=T_{q-1,s_2q}(u;x,R)$.
\end{Definition}
 Now, we define the notion of a local weak solution to problem \eqref{prob}.  
\begin{Definition}
	A function $u\in W^{s_1,p}_{\rm loc}(\Om)\cap L^{p-1}_{s_1p}(\mb R^N) \cap L^{q-1}_{s_2q}(\mb R^N)$ is said to be a local weak solution of problem \eqref{prob} if 
	\begin{align*}
	 &\int_{\mb R^N}\int_{\mb R^N}\frac{[u(x)-u(y)]^{p-1}}{|x-y|^{N+ps_1}}(\psi(x)-\psi(y))dxdy + \int_{\mb R^N}\int_{\mb R^N}\frac{[u(x)-u(y)]^{q-1}}{|x-y|^{N+qs_2}}(\psi(x)-\psi(y))dxdy \nonumber\\
	 &\quad= \int_{\Om} f\psi dx,
	\end{align*}
	for all $\psi\in W^{s_1,p}(\Om)$ compactly supported in $\Om$. 
\end{Definition}
Let $\Om\Subset\Om'\subset\mb R^N$. For   $g\in W^{s_1,p}(\Om')\cap L^{p-1}_{s_1p}(\mb R^N)\cap L^{q-1}_{s_2q}(\mb R^N)$, we consider the following prototype problem:
 \begin{equation*}
 	\left\{ \begin{array}{rlll}
 		(-\Delta)^{s_1}_{p}u+  (-\Delta)^{s_2}_{q}u &= f \quad \text{in} \; \Om, \\
 		u &=g \quad \text{in} \; \mb R^N\setminus \Om.
 	\end{array}
 	\right. \tag{$\mc G_{f,g}(\Om)$} \label{probMain}
 \end{equation*}
 To study the weak solutions of \eqref{probMain}, we state the following:
 \begin{Definition}
 	Let $\Om\Subset\Om'\subset\mb R^N$ and $1<p<\infty$ with $s_1\in (0,1)$, then we define
 	\begin{align*}
 		X^{s_1,p}_g(\Om,\Om'):= \{ v\in W^{s_1,p}(\Om')\cap L^{p-1}_{s_1p}(\mb R^N) : v=g \quad\mbox{a.e. in }\mb R^N\setminus\Om \},
 	\end{align*}
 	equipped with the norm of $W^{s_1,p}(\Om')$.
 \end{Definition}
 
 \begin{Definition}
  A function $u\in X^{s_1,p}_{g}(\Om,\Om') \cap X^{s_2,q}_{g}(\Om,\Om')$ is said to be a weak super-solution (resp. sub-solution) of \eqref{probMain} if $(g-u)_+\in X^{s_1,p}_{0}(\Om,\Om') \cap X^{s_2,q}_{0}(\Om,\Om')$ (resp. $(u-g)_+\in X^{s_1,p}_{0}(\Om,\Om') \cap X^{s_2,q}_{0}(\Om,\Om')$) and the following holds:
  \begin{align}\label{solndef}
  &\int_{\mb R^N}\int_{\mb R^N}\frac{[u(x)-u(y)]^{p-1}}{|x-y|^{N+ps_1}}(\phi(x)-\phi(y))dxdy + \int_{\mb R^N}\int_{\mb R^N}\frac{[u(x)-u(y)]^{q-1}}{|x-y|^{N+qs_2}}(\phi(x)-\phi(y))dxdy \nonumber\\
  &\quad\geq ({\rm resp. }\leq) \int_{\Om} f(x)\phi(x) dx,
  \end{align}
  for all non-negative $\phi\in X^{s_1,p}_{0}(\Om,\Om') \cap X^{s_2,q}_{0}(\Om,\Om')$. 
 \end{Definition}
Our first main theorem is the local H\"older continuity result for local weak solutions to problem \eqref{prob}. Precisely, we have: 
\begin{Theorem}\label{intreg}
 Suppose $1<q<p<\infty$. Let $u\in W^{s_1,p}_{\rm loc}(\Om)\cap L^{p-1}_{s_1p}(\mb R^N)\cap L^{q-1}_{s_2q}(\mb R^N)$ be a local weak solution to problem \eqref{prob}. Then, $u\in C^{0,\al}_{\rm loc}(\Om)$, for some $\al\in (0,1)$ satisfying $\al<\min \big\{ \frac{\ga s_1p-N}{\ga(p-1)}, \frac{qs_2}{q-1}, \frac{ps_1-qs_2}{p-q} \big\}$. Moreover, for any $R_0\in (0,1)$ such that $B_{2R_0}\equiv B_{2R_0}(x_0)\Subset\Om$, the following holds: 
	\begin{align*}
		[u]_{C^{0,\al}(B_r)} \leq  {\frac{c}{R_0^\al}} \Big[& \| u \|_{L^\infty(B_{R_0})}+R_0^\frac{ps_1-qs_2}{p-q}+  R_0^\frac{\ga s_1p-N}{\ga(p-1)} \| f \|_{L^\ga(B_{R_0})}^\frac{1}{p-1}+ T_{p-1}(u;x_0,R_0) \\
		&+ T_{q-1}(u;x_0,R_0)\Big],
	\end{align*}
	for all $r\in (0, R_0)$, where $c=c(N,s_1,p,s_2,q)>0$ is a constant.
\end{Theorem}
It is clear from the statement of Theorem \ref{intreg} that the H\"older exponent $\al$ is not optimal. However, for the case $q\geq 2$ and $f\in L^\infty_{\rm loc}(\Om)$, optimal value of $\al$ is mentioned in \cite{JDS2}, whereas the optimality of $\al$ in the subquadratic case is still unknown, even for the homogeneous operator case, that is, $p=q$ and $s_1=s_2$.
\begin{Remark}
\begin{itemize}
\item[(a)] For the case $q=p$ and $s_2< s_1$, the term  $R_0^\frac{ps_1-qs_2}{p-q}$ does not appear in Theorem \ref{intreg}, and the proof runs analogously. Indeed, in all of the technical results of Section 3, the term $\la^{q-p}r^{s_1p-s_2q}$ disappears because of the $p$-homogeneity.
\item[(b)] For the case $1<p<q<\infty$ and $s_1p\leq s_2 q$, we can proceed analogously by interchanging the role of $(p,s_1)$ with $(q,s_2)$, to get a similar result as in Theorem \ref{intreg}. However, for the case $s_1p>s_2 q$, we need to assume the boundedness of the weak solution in whole $\mb R^N$. Then, employing the Caccioppoli type inequality \cite[Lemma 3.1]{JDS2} and an analogous inequality of \eqref{eq11}, we can proceed as in \cite[Theorem 1.2]{dicastro} to get our interior H\"older continuity result.
\end{itemize}
 \end{Remark}

 \begin{Remark}\label{remgenker}
	As a matter of fact, the result of Theorem \ref{intreg} is valid for equations of the type \eqref{prob} involving a more general class of operators, for instance, 
	\begin{equation*}
		\mc L_{K_{\el,s}} u(x)= 2\ds\lim_{\e\ra 0}\int_{\mb R^N\setminus B_\e(x)} [u(x)-u(y)]^{\el-1}K_{\el,s}(x,y)dy,
	\end{equation*}
	where $(\el,s)\in\{(p,s_1),(q,s_2)\}$ with $1<q<p<\infty$ and $0<s_2\le s_1<1$. Here, the singular kernel $K_{\el,s}: \mb R^N\times\mb R^N\to[0,\infty)$ is such that 
	\begin{enumerate}
		\item [(i)] there exist $1\le c_p\leq C_p$ satisfying $c_p \leq K_{p,s_1}(x,y) |x-y|^{N+ps_1} \le C_p $, for a.a. $x,y\in \mb R^N$,
		\item[(ii)] there exist $0\leq c_q\leq C_q$ satisfying $ c_q  \leq K_{q,s_2}(x,y)|x-y|^{N+qs_2} \leq C_q $,  for a.a. $x,y\in \mb R^N$.
	\end{enumerate} 
	For example, one can take $K_{\el,s}(x,y)= a_\el(x,y) |x-y|^{-(N+\el s)}$, where $a_\el:\mb R^N\times\mb R^N\to\mb R$ are non-negative bounded functions, for $\el\in\{p,q\}$, with $\inf_{\mb R^{N}\times\mb R^N}a_p\geq 1$.
\end{Remark}
Next, we establish the following boundary regularity result.
 \begin{Corollary}[Boundary regularity]\label{bdryreg}
	Let $u\in  W^{s_1,p}_0(\Om)$ be a weak solution to \eqref{probMain} with $f\in L^\infty(\Om)$ and $g\equiv 0$. Then, there exists $\al\in (0,1)$ such that $u\in C^{0,\al}(\ov\Om)$. Moreover, 
	\begin{align}\label{holderbd}
		\| u \|_{C^{0,\al}(\ov\Om)} \leq C,
	\end{align}
	where $C=C(\Om,N,s_1,p,s_2,q, \|f\|_{L^\infty(\Om)})>0$ is a constant (which  depends as a non-decreasing function of $\|f\|_{L^\infty(\Om)}$).
\end{Corollary}
\begin{Remark}
 As a consequence of Corollary \ref{bdryreg}, we have the boundary regularity result for the critical exponent problem. Precisely, let $u\in W^{s_1,p}_0(\Om)$ be a solution to problem \eqref{probMain} with $f(x):=f(x,u)$ and $g\equiv 0$, where $f$ is a Carath\'eodory function such that $|f(x,t)|\leq C_0 (1+|t|^{p^*_{s_1}-1})$, with $C_0 >0$, and $p^*_{s_1}:=Np/(N-ps_1)$ if $N>ps_1$, otherwise an arbitrarily large number. Then, $u\in C^{0,\al}(\ov\Om)$, for some $\al\in (0,s_1]$, and \eqref{holderbd} holds. The proof follows by noting the fact that $u\in L^\infty(\Om)$, and hence $f\in L^\infty(\Om)$. 
\end{Remark}
\begin{Remark}\label{remmixed}
We further observe that for the case $0<s_2<s_1=1$, the problem \eqref{prob} exhibits a nonhomogeneous local-nonlocal behavior, see for instance \cite{BDVV,CKPV} for corresponding semilinear homogeneous counterparts. In this case, proceeding similarly we can prove an analogous result to Theorem \ref{intreg}. Subsequently, following the proofs of \cite[Lemma 2.2]{JDS} (without the approximation argument) and \cite[Proposition 3.11]{JDS2}, we obtain the boundary behavior of the weak solution. Consequently, we get the global H\"older continuity result as in Corollary \ref{bdryreg}.
\end{Remark}
As of independent interest, we have the following Harnack inequality for weak solutions.
\begin{Theorem}[Harnack inequality]\label{harnack}
	Let  $u\in X^{s_1,p}_{g}(\Om,\Om') \cap X^{s_2,q}_{g}(\Om,\Om')$ be a weak solution to problem \eqref{probMain} such that $u\geq 0$ in $B_R\equiv B_R(x_0)\Subset\Om$, for some $R\in (0,1)$. Then,  for all $0<r<R$, the following holds:
	\begin{align*}
		\sup_{B_{r/2}} u &\leq C \inf_{B_{r/2}} u+ C\Big(\frac{r}{R}\Big)^\frac{s_1p}{p-1} T_{p-1}(u_-;x_0,R) + C r^\frac{s_1p-s_2q}{p-1} \Big(\frac{r}{R}\Big)^\frac{s_2q}{p-1} T_{q-1}(u_-;x_0,R)^\frac{q-1}{p-1} \\
		&\quad+ C r^\frac{s_1p-s_2q}{p-q} +C\|f\|_{L^\ga(B_R)}^{1/(p-1)}, 
	\end{align*}
	where  $C=C(N,p,q,s_1,s_2)>0$ is a constant.
\end{Theorem}
Next, we have the following strong maximum principle.
\begin{Theorem}\label{strngmax}
	Suppose $1<q\leq p<\infty$. Let $g\in C(\mb R)\cap BV_{\rm loc}(\mb R)$ and let $u\in W^{s_1,p}_0(\Om)\cap C(\ov\Om)$ be such that 
	\begin{align*}
		(-\De)_p^{s_1} u+ (-\De)_q^{s_2} u +g(u)\geq g(0) \quad\mbox{weakly in }\Om.
	\end{align*}
	Further, assume that $u\not\equiv 0$ with $u\geq 0$ in $\Om$.  Then, there exists $c_1>0$ such that $u\geq c_1 {\rm dist}(\cdot,\pa\Om)^{s_1}$ in $\Om$.
\end{Theorem}

 \section{Some technical results}
 In this section, we prove some preliminary results such as Caccioppoli type inequality, local boundedness result and expansion of positivity result, which are required to prove our main theorems.
 \begin{Lemma}\label{lem1}
 Let  $u\in X^{s_1,p}_{g}(\Om,\Om') \cap X^{s_2,q}_{g}(\Om,\Om')$ be a super-solution to problem \eqref{probMain} such that $u\geq 0 $ in $B_R\equiv B_R(x_0)\subset\Om$, for $R\in (0,1)$. For $k\geq 0$, suppose that there exists $\nu\in(0,1]$ such that
 \begin{align}\label{eq10}
 	\frac{|B_r \cap \{ u \ge k \}|}{|B_r|} \geq \nu \quad\mbox{for all } r\in (0,R/16).
 \end{align}
Then, there exists a constant $c=c(N,s_1,p,s_2,q)>0$ such that, for all $\de\in (0,1/4)$:
 \begin{align*}
  \bigg|B_{6r} \cap &\left\{ u \leq 2\de k-\frac{1}{2} \Big(\frac{r}{R}\Big)^\frac{s_1p}{p-1}T_{p-1}(u_-;x_0,R) -\frac{1}{2}\Big(\frac{r}{R}\Big)^\frac{s_2q}{q-1}T_{q-1}(u_-;x_0,R)
  -\frac{1}{2} r^\frac{s_1p-s_2q}{p-q} \right. \\
 &\left.  -\frac{1}{2} r^\frac{\ga s_1p-N}{\ga(p-1)} \| f \|_{L^\ga(B_R)}^\frac{1}{p-1}   \right\}\bigg| 
  \leq \frac{c \; |B_{6r}|}{\nu \log(\frac{1}{2\de})}.
 \end{align*}
\end{Lemma} 
\begin{proof}
 Let $\la>0$ be a parameter (to be chosen later) and set $\bar{u}= u+\la$. Then, $\bar{u}$ is still a super-solution to \eqref{probMain}. Let $\phi\in C_c^\infty(B_{7r})$ be such that $0\leq\phi\leq 1$, $|\na \phi|\leq c/r$ in $B_{7r}$ and $\phi\equiv 1$ in $B_{6r}$. We take $\bar{u}^{1-p}\phi^p$ as a test function in the weak formulation \eqref{solndef}, to get
  \begin{align*}
 	0 &\leq \sum_{(\el,s)}\int_{B_{8r}} \int_{B_{8r}}  \frac{[\bar u(x)-\bar u(y)]^{\el-1}}{|x-y|^{N+s\el}}\left[\frac{\phi^p(x)}{\bar u(x)^{p-1}}-\frac{\phi^p(y)}{\bar u(y)^{p-1}}\right]dxdy 
 		\\
 	&\quad  + 2 \sum_{(\el,s)}\int_{{\mb R^N}\setminus B_{8r}}\int_{B_{8r}} \frac{[\bar u(x)-\bar u(y)]^{\el-1}}{|x-y|^{N+s\el}} \frac{\phi^p(x)}{\bar u(x)^{p-1}}dxdy
 	 - \int_{\Om} f(x) \bar{u}^{1-p}(x)\phi^p(x)dx  \\ 
 		&=:J_1(\el)+J_2(\el)-J_3(f). 
 \end{align*}
 We estimate the quantities $J_1$, $J_2$ and $J_3$ in the following steps:\\
 \textit{Step I}: Estimate of $J_1(p)$.\\
 Following the proof of \cite[Lemma 1.3, (3.12) and (3.17)]{dicastro}, we have 
 \begin{align*}
 	J_1(p)&\leq -c \int_{B_{8r}} \int_{B_{6r}} |x-y|^{-N-s_1p} \bigg|\log\left(\frac{\bar u(x)}{\bar u(y)}\right)\bigg|^p dxdy + c \int_{B_{8r}} \int_{B_{8r}}  \frac{|\phi(x)-\phi(y)|^p}{|x-y|^{N+s_1p}} dxdy \\
 	&\leq -c \int_{B_{8r}} \int_{B_{6r}} |x-y|^{-N-s_1p} \bigg|\log\left(\frac{\bar u(x)}{\bar u(y)}\right)\bigg|^p dxdy + c r^{N-s_1p}.
 \end{align*}
 \textit{Step II}: Estimate of $J_1(q)$.\\
 Without loss of generality, we assume that $u(x)> u(y)$. Recall the following inequality \cite[Lemma 3.1]{dicastro}: For $p\geq 1$, $\e\in (0,1]$ and for all $a,b\in \mb R^N$, we have 
 \begin{align}\label{ineqlem3.1}
 	|a|^p\leq (1+c_p\e) |b|^p + (1+c_p\e)  \e^{1-p} |a-b|^p, \quad\mbox{where }c_p:=(p-1)\Ga(\max\{1,p-2\}).
 \end{align}
For the choice $p=q$, $a=\phi^{p/q}(x)$, $b=\phi^{p/q}(y)$ together with
 \begin{align*}
 	\epsilon = t\frac{\bar u(x)-\bar u(y)}{\bar u(x)} \in (0,1) \quad\mbox{with } t\in(0,1),
 \end{align*}
  in \eqref{ineqlem3.1} (note that $u\geq 0$ in $B_{8r}$), we obtain
 \begin{align*}
 	\phi^p(x) \leq \left(1+c t\frac{\bar u(x)-\bar u(y)}{\bar u(x)}\right)\phi^{p}(y) + (1+c \e)t^{1-q} \left(\frac{\bar u(x)-\bar u(y)}{\bar u(x)}\right)^{1-q} |\phi^{p/q}(x)-\phi^{p/q}(y)|^q.
 \end{align*}
 This implies that 
 \begin{align*}
 	(\bar u(x)-\bar u(y))^{q-1}  \frac{\phi^p(x)}{\bar u(x)^{p-1}} \leq& \left(1+c t \frac{\bar u(x)-\bar u(y)}{\bar u(x)} \right) \frac{(\bar u(x)-\bar u(y))^{q-1}}{\bar u(x)^{p-1}} \phi^p(y) \\
 	&+(1+c) t^{1-q}\bar u(x)^{q-p}|\phi^{p/q}(x)-\phi^{p/q}(y)|^q.
 \end{align*}
 Now, using the relation $
 |\phi(x)^{p/q}-\phi(y)^{p/q}| \leq \frac{p}{q}(\phi(x)^{p/q}+\phi(y)^{p/q})^{(p-q)/p} |\phi(x)-\phi(y)|$, we get
 {\small \begin{align*}
 		(\bar u(x)-\bar u(y))^{q-1} \left[ \frac{\phi^p(x)}{\bar u(x)^{p-1}} - \frac{\phi^p(y)}{\bar u(y)^{p-1}}\right] 
 		&\leq \phi^p(y) \frac{(\bar u(x)-\bar u(y))^{q-1}}{\bar u(x)^{p-1}} \left[ 1+c t \frac{\bar u(x)-\bar u(y)}{\bar u(x)}   -\frac{\bar u(x)^{p-1}}{\bar u(y)^{p-1}} \right] \\
 		&\quad +c t^{1-q} \la^{q-p} |\phi(x)-\phi(y)|^q,
 \end{align*}}
 \noi where in the last inequality we have used the relation $\bar u\geq \la$ in $B_{R}$. The terms inside the bracket on the right hand side of the above expression is similar to the one of \cite[(3.6)]{dicastro}.  Therefore,  (see \cite[(3.9)]{dicastro} and the expression for $g(z)\leq -p+1$ there)
 \begin{align*}
 	 \left[ 1+c t \frac{\bar u(x)-\bar u(y)}{\bar u(x)}   -\frac{\bar u(x)^{p-1}}{\bar u(y)^{p-1}} \right] \leq 0.
 \end{align*}
 And an analogous result holds for the case $u(x)<u(y)$. Thus 
 \begin{align*}
 	J_1(q) \leq c \la^{q-p} \int_{B_{8r}} \int_{B_{8r}} |x-y|^{-N-s_2q} |\phi(x)-\phi(y)|^q dxdy 
 	\leq c \la^{q-p} r^{N-s_2q}.
 \end{align*}
 \textit{Step III}: Estimate of $J_2(p)$.\\
 From \cite[Lemma 3.1, p.1818]{dicastroHarn}, we have 
 \begin{align*}
 	J_2(p) &\leq  c r^{N-s_1p}+ c \la^{1-p} r^N R^{-s_1p} T_{p-1}(u_-;x_0,R)^{p-1}.
 \end{align*}
 \textit{Step IV}: Estimate of $J_2(q)$.\\
 Noting $u\geq 0$ in $B_R$, we have 
 \begin{align*}
 	&\frac{u(x)-u(y)}{u(x)+\la} \leq 1 \mbox{ for all }x\in B_{8r}, \ y\in B_R, \\
 	&(u(x)-u(y))_+^{q-1} \leq 2^{q-1} [ u(x)^{q-1}+(u(y))_-^{q-1} ] \ \mbox{ for all }x\in B_{8r}, \ \ y\in \mb R^N\setminus B_{R}.
 \end{align*}
 Therefore,
 \begin{align*}
 	&J_2(q)\nonumber\\
 	 &\leq 2 \left( \int_{B_R\setminus B_{8r}} \int_{B_{8r}} + \int_{\mb R^N\setminus B_{R}} \int_{B_{8r}} \right) |x-y|^{-N-s_2q} \frac{(u(x)-u(y))_+^{q-1}}{(u(x)+\la)^{p-1}}\phi^p(x) dxdy \\
 	&\leq \la^{q-p} \int_{\mb R^N\setminus B_{8r}} \int_{B_{8r}}  \frac{\phi^p(x)}{|x-y|^{N+s_2q}}dxdy + c \int_{\mb R^N\setminus B_{R}} \int_{B_{8r}} \frac{u(x)^{q-1}+(u(y))_-^{q-1}}{(u(x)+\la)^{p-1}}\frac{\phi^p(x)dxdy}{|x-y|^{N+s_2q}} \\
 	&\leq \la^{q-p} \int_{\mb R^N\setminus B_{8r}} \int_{B_{7r}} \frac{\phi^p(x)}{|x-y|^{N+s_2q}}dxdy + c \la^{1-p} \int_{\mb R^N\setminus B_{R}} \int_{B_{7r}} \frac{(u(y))_-^{q-1}\phi^p(x)}{|x-y|^{N+s_2q}} dxdy \\
 	&\leq cr^{N}\la^{q-p} \int_{\mb R^N\setminus B_{8r}}  \frac{1}{|x_0-y|^{N+s_2q}}dy + c r^N \la^{1-p} \int_{\mb R^N\setminus B_{R}}  \frac{(u(y))_-^{q-1}}{|x_0-y|^{N+s_2q}} dy,
 \end{align*}
that is, 
 \begin{align*}
 	J_2(q)\leq \la^{q-p} c r^{N-s_2q} + c r^N \la^{1-p} R^{-s_2q}  T_{q-1}(u_-;x_0,R)^{q-1}.
 \end{align*}
\textit{Step V}: Estimate of $J_3(f)$.\\
 Noting the bounds on $\bar u$ and $\phi$, we have 
 \begin{align*}
	|-J_3(f)|\leq \int_{\Om} |f(x)| \bar{u}^{1-p}(x)\phi^p(x)dx \leq \la^{1-p} \int_{B_{7r}} |f|\phi^p &\leq c \la^{1-p}\| f \|_{L^\ga(B_R)} r^{N/\ga'}. 
\end{align*}
Combining the estimates of steps I through V, we obtain
{\small\begin{align}\label{eq11}
		\int_{B_{6r}} \int_{B_{6r}} \bigg|\log\left(\frac{\bar u(x)}{\bar u(y)}\right)\bigg|^p \frac{dxdy}{|x-y|^{N+s_1p} }
		&\leq c r^{N-s_1p} + c \la^{q-p} r^{N-s_2q}+c \la^{1-p} r^N R^{-s_1p} T_{p-1}(u_-;x_0,R)^{p-1} \nonumber\\
		&\ +c r^N \la^{1-p} R^{-s_2q}  T_{q-1}(u_-;x_0,R)^{q-1}+c \la^{1-p}\| f \|_{L^\ga(B_R)} r^{N/\ga'}.
\end{align}}
Next, we take 
\begin{align}\label{eqlamb}
	\la:= \Big(\frac{r}{R}\Big)^\frac{s_1p}{p-1}T_{p-1}(u_-;x_0,R) +\Big(\frac{r}{R}\Big)^\frac{s_2q}{q-1}T_{q-1}(u_-;x_0,R) +r^\frac{s_1p-s_2q}{p-q}+\| f \|_{L^\ga(B_R)}^\frac{1}{p-1} r^\frac{\ga s_1p-N}{\ga(p-1)}
\end{align}
 and note that $(\ga s_1p-N)/\ga=N/\ga'-N+s_1p > 0$ (thanks to the assumption on $\ga$).
Moreover, by noting $\la \geq r^\frac{s_1p-s_2q}{p-q}$, we observe that 
\begin{align*}
 \la^{1-p} r^N R^{-s_2q}  T_{q-1}(u_-;x_0,R)^{q-1} &= r^{N-s_1p}\la^{1-q} \Big(\frac{r}{R}\Big)^{s_2q} T_{q-1}(u_-;x_0,R)^{q-1} \la^{q-p}	r^{s_1p-s_2q} \\
 &\leq r^{N-s_1p}.
\end{align*}
Therefore, from \eqref{eq11}, taking into account \eqref{eqlamb},  we get 
 \begin{align}\label{eq12}
 	\int_{B_{6r}} \int_{B_{6r}} |x-y|^{-N-s_1p} \bigg|\log\left(\frac{\bar u(x)}{\bar u(y)}\right)\bigg|^p dxdy 
 	\leq c r^{N-s_1p}.
 \end{align}
Next, for any $\de\in (0,1/4)$, we set 
 \begin{align*}
 	v:= \bigg[\min\bigg\{ \log\frac{1}{2\de}, \log\frac{k+\la}{\bar u} \bigg\}\bigg]_+.
 \end{align*}
 Noting that $v$ is a truncation of the sum of a constant and $\log\bar u$, and using the fractional Poincar\'e type inequality, we have 
 \begin{align*}
 	\Xint-_{B_r} |v(x)-(v)_{B_r}|^p dx &\leq C r^{s_1p-N} \int_{B_{6r}}\int_{B_{6r}} \frac{|v(x)-v(y)|^p}{|x-y|^{N+s_1p}}dxdy \\
 	&\leq  Cr^{s_1p-N}\int_{B_{6r}} \int_{B_{6r}} |x-y|^{-N-s_1p} \bigg|\log\left(\frac{\bar u(x)}{\bar u(y)}\right)\bigg|^p dxdy \leq C,
 \end{align*}
where in the last inequality we have used \eqref{eq12}. This, on using H\"older's inequality, yields
 \begin{align}\label{eq13}
 	\Xint-_{B_r} |v(x)-(v)_{B_r}|dx \leq C.
 \end{align}
 By the definition of $v$, we observe that 
 \begin{align*}
 	\{v=0\}=\{\bar u \geq k+\la \}=\{ u\geq k\}.
 \end{align*}
Hence, from the assumption of the lemma (see \eqref{eq10}), we have 
 \begin{align*}
 	\frac{|B_{6r} \cap \{ v=0 \}|}{|B_{6r}|} \geq \frac{\nu}{6^N}.
 \end{align*}
Therefore, 
 \begin{align*}
	\log\frac{1}{2\de}=\frac{1}{| B_{6r}\cap\{ v=0 \}|}\int_{B_{6r}\cap\{ v=0 \}}\log\frac{1}{2\de} dx &\leq \frac{6^N}{\nu}\frac{1}{ |B_{6r}|} \int_{ B_{6r}} \Big(\log\frac{1}{2\de}-v(x)\Big)dx \\
	&=\frac{6^N}{\nu} \bigg[\log\frac{1}{2\de}-(v)_{B_{6r}}\bigg],
\end{align*}
which upon integration and using \eqref{eq13} implies that
\begin{align*}
	\frac{|B_{6r}\cap\{ v=\log\frac{1}{2\de} \}|}{|B_{6r}|}\log\frac{1}{2\de} &\leq \frac{6^N}{\nu|B_{6r}|}\int_{B_{6r}\cap\{ v=\log\frac{1}{2\de} \}} [\log\frac{1}{2\de}-(v)_{B_{6r}}] dx \nonumber \\
	&\leq\frac{2}{|B_{6r}|}\int_{B_{6r}}|v(x)-(v)_{B_{6r}}|dx 
	\leq C.
\end{align*}
Again, in view of the definition of $v$, for all $\de\in (0,1/4)$, we get 
\begin{align*}
	\frac{|B_{6r}\cap\{ \bar u \leq 2\de(k+\la) \}|}{|B_{6r}|} \leq \frac{c}{\nu}\frac{1}{\log\frac{1}{2\de}}.
\end{align*}
 This yields the required result of the lemma upon using the fact that $\bar u= u+\la$, where $\la$ is given by \eqref{eqlamb}.\QED
\end{proof}
By slightly modifying the proof of Lemma \ref{lem1}, we have the following result.
\begin{Lemma}
 Suppose $1<q\leq p<\infty$. Let $u\in W^{s_1,p}(\Om)\cap L^\infty_{\rm loc}(\Om)$ be a weak super-solution to $(-\De)_p^{s_1} u+ (-\De)_q^{s_2} u\geq 0$ in $\Om$ such that $u\ge 0$ in $B_R\equiv B_R(x_0)\subset\Om$, for some $R\in (0,1)$. Then, for all $r\in (0,R/4)$ and $\la>0$, the following holds:
	\begin{align}\label{eqA18}
		\Xint-_{B_r}\int_{B_r} \bigg|\log\frac{u(x)+\la}{u(y)+\la}\bigg|^q\frac{dxdy}{|x-y|^{N+s_2q}}
		&\leq c\la^{1-q}\big[ R^{-s_1p}T_{p-1}(u_-;R)^{p-1}+ R^{-s_2q}T_{q-1}(u_-;R)^{q-1}\big] \nonumber\\
		&\quad+ c r^{-s_2q}+cr^{-s_1p}(\|u\|_{L^\infty(B_R)}+\la)^{p-q}.     
	\end{align}
\end{Lemma}
\begin{proof}
	Set $\ov u=u+\la$ and take $\ov u(x)^{1-q}\phi(x)^p$ as a test function, where $\phi\in C_c^\infty(B_{3r/2})$ is such that $0\leq \phi\le 1$, $|\na \phi|\leq c/r$ in $B_{3r/2}$ and $\phi\equiv 1$ in $B_r$. Thus,
	\begin{align}\label{eqA10}
		0 &\leq \sum_{(\el,s)}\int_{B_{2r}} \int_{B_{2r}}  \frac{[\bar u(x)-\bar u(y)]^{\el-1}}{|x-y|^{N+s\el}}\left[\frac{\phi^p(x)}{\bar u(x)^{q-1}}-\frac{\phi^p(y)}{\bar u(y)^{q-1}}\right]dxdy 
		\nonumber\\
		&\quad  + 2 \sum_{(\el,s)}\int_{\mb R^N\setminus B_{2r}}\int_{B_{2r}} \frac{[\bar u(x)-\bar u(y)]^{\el-1}}{|x-y|^{N+s\el}} \frac{\phi^p(x)}{\bar u(x)^{q-1}}dxdy \nonumber\\ 
		&=:I_1(\el)+I_2(\el). 
	\end{align}
	We estimate the quantities $I_1$ and $I_2$ in the following steps:\\
	\textit{Step I}: Estimate of $I_1(p)$.\\
	Without loss of generality, we assume $u(x)> u(y)$. Employing \eqref{ineqlem3.1}, for the choice $a=\phi(x)$, $b=\phi(y)$ together with
	\begin{align*}
		\epsilon = t\frac{\bar u(x)-\bar u(y)}{\bar u(x)} \in (0,1) \quad\mbox{with } t\in(0,1),
	\end{align*}
	(note that $u\geq 0$ in $B_{2r}$), we obtain (upon simplification)
	{\small\begin{align*}
			(\bar u(x)-\bar u(y))^{p-1} \left[ \frac{\phi^p(x)}{\bar u(x)^{q-1}} - \frac{\phi^p(y)}{\bar u(y)^{q-1}}\right] 
			&\leq \phi^p(y) \frac{(\bar u(x)-\bar u(y))^{p-1}}{\bar u(x)^{q-1}} \left[ 1+c t \frac{\bar u(x)-\bar u(y)}{\bar u(x)}   -\frac{\bar u(x)^{q-1}}{\bar u(y)^{q-1}} \right] \\
			&\quad +c t^{1-p} (u(x)+\la)^{p-q} |\phi(x)-\phi(y)|^p.
	\end{align*}}
	Therefore, as in Lemma 3.1 (the first term on the r.h.s. is non-positive), we get 
	\begin{align}\label{eqA12}
		I_1(p)\leq c \int_{B_{2r}}\int_{B_{2r}} (u(x)+\la)^{p-q} \frac{|\phi(x)-\phi(y)|^p}{|x-y|^{N+s_1p}}dxdy\leq c(\|u\|_{L^\infty(B_R)}+\la)^{p-q} r^{N-s_1p}.
	\end{align}
	\textit{Step II}: Estimate of $I_1(q)$.\\
	From \eqref{ineqlem3.1} for the choice $p=q$, $a=\phi^{p/q}(x)$, $b=\phi^{p/q}(y)$ together with $\e$ as in step I, we get 
	{\small \begin{align*}
			(\bar u(x)-\bar u(y))^{q-1} \left[ \frac{\phi^p(x)}{\bar u(x)^{q-1}} - \frac{\phi^p(y)}{\bar u(y)^{q-1}}\right] 
			&\leq \phi^p(y) \frac{(\bar u(x)-\bar u(y))^{q-1}}{\bar u(x)^{q-1}} \left[ 1+c t \frac{\bar u(x)-\bar u(y)}{\bar u(x)}   -\frac{\bar u(x)^{q-1}}{\bar u(y)^{q-1}} \right] \\
			&\quad +c t^{1-q} |\phi^{p/q}(x)-\phi^{p/q}(y)|^q.
	\end{align*}}
	Following the proof of \cite[Lemma 1.3, (3.12) and (3.17)]{dicastro}, we have 
	\begin{align}\label{eqA13}
		I_1(q)\leq -c_2 \int_{B_{2r}}\int_{B_{2r}}  \bigg|\log\frac{u(x)+\la}{u(y)+\la}\bigg|^q\frac{\phi^{p}(y)dxdy}{|x-y|^{N+s_2q}}+ c \int_{B_{2r}}\int_{B_{2r}}  \frac{|\phi(x)-\phi(y)|^q}{|x-y|^{N+s_2q}}dxdy.
	\end{align}
	\textit{Step III}: Estimate of $I_2(\el)$.\\
	On a similar note to Lemma 3.1, we can prove the following:
	\begin{equation}\label{eqA14}
		\begin{aligned}
			&I_2(p)\leq c r^{N-s_1p} (\|u\|_{L^\infty(B_R)}+\la)^{p-q}+ c \la^{1-q} r^N R^{-s_1p} T_{p-1}(u_-;R)^{p-1} \quad\mbox{and} \\
			&I_2(q)\leq c r^{N-s_2q} + c \la^{1-q} r^N  R^{-s_2q}T_{q-1}(u_-;R)^{q-1}.
		\end{aligned}
	\end{equation}
	Combining \eqref{eqA12}, \eqref{eqA13} and \eqref{eqA14} with \eqref{eqA10}, and noting $\phi\equiv 1$ in $B_r$, we get the required result of the lemma. \QED
\end{proof}

Now, we have the expansion of positivity result as below.
\begin{Lemma}\label{lem2}
 Suppose that the hypotheses of Lemma \ref{lem1} hold true. 
 Then, there exists a constant $\de=\de(N,s_1,p,s_2,q)\in(0,1/4)$ such that 
 \begin{align*}
  \inf_{B_{4r}} u \geq& \de k  - 
  r^\frac{\ga s_1p-N}{\ga(p-1)} \| f \|_{L^\ga(B_R)}^\frac{1}{p-1}- r^\frac{s_1p-s_2q}{p-q}-\sum_{(\el,s)} \Big(\frac{r}{R}\Big)^\frac{s\el}{\el-1}T_{\el-1}(u_-;x_0,R).
 \end{align*}	
\end{Lemma} 
\begin{proof}
 Due to the fact $u\geq 0$ in $B_R$, we may assume that (otherwise there is noting to prove)
 \begin{align}\label{eq14}
 	\Big(\frac{r}{R}\Big)^\frac{s_1p}{p-1}T_{p-1}(u_-;x_0,R) +\Big(\frac{r}{R}\Big)^\frac{s_2q}{q-1}T_{q-1}(u_-;x_0,R)+ r^\frac{s_1p-s_2q}{p-q}+  r^\frac{\ga s_1p-N}{\ga(p-1)} \| f \|_{L^\ga(B_R)}^\frac{1}{p-1}\leq \de k.
 \end{align}
Set $w_-=(b-u)_+$, for $b>0$. For any $r \leq \rho \leq 6r$ and fixed $\psi\in C_c^\infty(B_\rho)$ with $0\leq \psi\leq 1$, we take $w_-\psi^p$  as a test function in the weak formulation, to get 
 \begin{align*}
	0&\leq \sum_{(\el,s)} \int_{B_\rho}\int_{B_\rho} \frac{[u(x)-u(y)]^{\el-1} (w_-(x)\psi(x)^p-w_-(y)\psi(y)^p) }{|x-y|^{N+\el s}} dxdy \nonumber\\
	&\quad+2\sum_{(\el,s)} \int_{\mb R^N\setminus B_\rho}\int_{B_\rho} \frac{[u(x)-u(y)]^{\el-1}w_-(x)\psi(x)^p}{|x-y|^{N+\el s}} dxdy-\int_{\Om} f(x)w_-(x)\psi^p(x)dx \nonumber\\  
	&=:I_1(\el)+I_2(\el)-I_3(f).
\end{align*}
 To estimate $I_1(\el)$, we first observe that
 \[ [u(x)-u(y)]^{\el-1} (w_-(x)\psi(x)^p-w_-(y)\psi(y)^p) \leq -[w_-(x)-w_-(y)]^{\el-1} (w_-(x)\psi(x)^p-w_-(y)\psi(y)^p),\] 
 for all $x,y\in B_\rho$. Therefore, proceeding similarly to the proof of \cite[Lemma 3.1]{JDS2}, we have 
\begin{align*}
 I_1(\el) &\leq -\frac{1}{4}\int_{B_\rho}\int_{B_\rho} \frac{|w_-(x)-w_-(y)|^l}{|x-y|^{N+s\el}} \big(\psi(x)^p+\psi(y)^p\big) dxdy \\
 &\quad+C  \int_{B_\rho}\int_{B_\rho} \frac{|\psi(x)-\psi(y)|^\el}{|x-y|^{N+s\el}} \big(w_-(x)^p+w_-(y)^p\big)dxdy.	
 \end{align*}
 Now, using 
 \begin{align*}
 	|w_-(x)\psi(x) - w_-(y)\psi(y)|^p &\leq 2^{p-1} |w_-(x) - w_-(y)|^p (\psi(x)^p+\psi(y)^p) \nonumber \\
 	&\quad+ 2^{p-1} |\psi(x)-\psi(y)|^{p}(w_-(x)+w_-(y))^p,
 \end{align*}
 we obtain 
 \begin{align}\label{eq16}
 	I_1(p)+I_1(q) &\leq -c \int_{B_\rho}\int_{B_\rho} \frac{|w_-(x)\psi(x) - w_-(y)\psi(y)|^p}{|x-y|^{N+s_1p}} dxdy \nonumber\\
 	&\quad+ C\sum_{(\el,s)}   \int_{B_\rho}\int_{B_\rho} \frac{|\psi(x)-\psi(y)|^\el}{|x-y|^{N+s\el}} \big(w_-(x)^p+w_-(y)^p\big)dxdy.	
 \end{align}
Next, we estimate $I_2$, similarly to \cite[Lemma 3.2]{dicastroHarn}, as below
 \begin{align*}
 	I_2(\el) &= \left( \int_{(\mb R^N\setminus B_\rho)\cap \{u(y)<0\} }\int_{B_\rho} + \int_{(\mb R^N\setminus B_\rho)\cap \{u(y)\geq 0\} }\int_{B_\rho} \right) \frac{[u(x)-u(y)]^{\el-1}w_-(x)\psi(x)^p}{|x-y|^{N+\el s}} dxdy \\
 	&=: I_{2,1}+ I_{2,2}.
 \end{align*}
 First we note that 
 \begin{align*}
 	\frac{[u(x)-u(y)]^{\el-1}}{|x-y|^{N+s\el}} w_-(x)\psi(x)^p \leq (b+u(y)_-)^{\el-1} b \Big(\sup_{x\in {\rm supp}\psi} |x-y|^{-N-s\el} \Big) \chi_{B_\rho\cap\{u<b\}}(x).
 \end{align*}
Therefore,
\begin{align*}
 I_{2,1}(\el) \leq b \Big(\sup_{x\in {\rm supp}\psi} \int_{\mb R^N\setminus B_\rho} (b+u(y)_-)^{\el-1}|x-y|^{-N-s\el}dy \Big) |B_\rho\cap\{u<b\}|.
\end{align*}
For $I_{2,2}$, we observe that $u\geq 0$ in $B_\rho$, thus proceeding similarly, we obtain
 \begin{align}\label{eq17}
 	I_2(\el)=I_{2,1}(\el)+I_{2,2}(\el) \leq c b \Big(\sup_{x\in {\rm supp}\psi} \int_{\mb R^N\setminus B_\rho} (b+u(y)_-)^{\el-1}|x-y|^{-N-s\el}dy \Big) |B_\rho\cap\{u<b\}|.
 \end{align}
For $I_3(f)$, by observing that $w_-\leq b$ and applying H\"older's inequality, we have
 \begin{align}\label{eq17f}
 	|-I_3(f)| \leq b \| f \|_{L^\ga(B_R)} |B_\rho\cap\{u<b\}|^{1/\ga'}.
 \end{align}
 Therefore, taking into account \eqref{eq14}, \eqref{eq16}, \eqref{eq17} and \eqref{eq17f}, we deduce that 
\begin{align}\label{eq18}
 	&\int_{B_\rho}\int_{B_\rho} \frac{|w_-(x)\psi(x) - w_-(y)\psi(y)|^p}{|x-y|^{N+s_1p}}dxdy  \nonumber\\
 	&\leq C\sum_{(\el,s)}   \int_{B_\rho}\int_{B_\rho} \frac{|\psi(x)-\psi(y)|^\el}{|x-y|^{N+s\el}} \big(w_-(x)^p+w_-(y)^p\big)dxdy \nonumber \\
 	& \quad+ Cb \sum_{(\el,s)} \Big(\sup_{x\in {\rm supp}\psi} \int_{\mb R^N\setminus B_\rho} \frac{(b+u(y)_-)^{\el-1}}{|x-y|^{N+s\el}}dy \Big)|B_\rho\cap\{u<b\}| +b  \| f \|_{L^\ga(B_R)} |B_\rho\cap\{u<b\}|^\frac{1}{\ga'}.
 \end{align}
Now, we start the iteration process to conclude the proof. Let 
 \begin{align*}
 	b\equiv b_j:= \de k + 2^{-j-1}\de k, \ \rho\equiv \rho_j:= 4r+2^{1-j}r \ \mbox{ and }\bar{\rho}_j=\frac{\rho_j+\rho_{j+1}}{2},
 \end{align*}
for all $j\in \mb N\cup\{0\}$. Observe that $\rho_j$ and $\bar{\rho}_j\in (4r,6r)$ and 
 \begin{align*}
 	b_j - b_{j+1} = 2^{-j-2}\de k \geq 2^{-j-3}b_j \quad\mbox{for all } j\in\mb N\cup\{0\}.
 \end{align*}
On account of \eqref{eq14}, we have 
\begin{align*}
	b_0=\frac{3}{2}\de k\leq &2\de k -\frac{1}{2} \Big(\frac{r}{R}\Big)^\frac{s_1p}{p-1}T_{p-1}(u_-;x_0,R) -\frac{1}{2}\Big(\frac{r}{R}\Big)^\frac{s_2q}{q-1}T_{q-1}(u_-;x_0,R)-\frac{1}{2} r^\frac{s_1p-s_2q}{p-q} \\
	&-\frac{1}{2} \| f \|_{L^\ga(B_R)}^\frac{1}{p-1} r^\frac{\ga s_1p-N}{\ga(p-1)} ,
\end{align*}
 which implies that 
 \begin{align*}
 	\{ u <b_0\} \subset  &\left\{ u < 2\de k-\frac{1}{2} \Big(\frac{r}{R}\Big)^\frac{s_1p}{p-1}T_{p-1}(u_-;x_0,R) -\frac{1}{2}\Big(\frac{r}{R}\Big)^\frac{s_2q}{q-1}T_{q-1}(u_-;x_0,R)-\frac{1}{2} r^\frac{s_1p-s_2q}{p-q} \right.\\
 	&\left. \quad-\frac{1}{2} \| f \|_{L^\ga(B_R)}^\frac{1}{p-1} r^\frac{\ga s_1p-N}{\ga(p-1)} \right\}.
 \end{align*} 
Therefore, Lemma \ref{lem1} gives us
\begin{align}\label{eqlev0}
	\frac{|B_{6r} \cap \{ u <b_0 \}|}{|B_{6r}|}\leq \frac{\tl c}{\nu \log(\frac{1}{2\de})}.
\end{align}
Set $w_-\equiv w_j=(b_j-u)_+$ and $B_j:= B_{\rho_j}(x_0)$. Then, we have 
\[  w_j \geq (b_j-b_{j+1}) \chi_{\{u<b_{j+1}\}} \geq 2^{-j-3} b_j \chi_{\{u<b_{j+1}\}}. \]
 Let $\psi_j\in C_c^\infty(B_{\bar{\rho}_j})$ be such that $0\le \psi_j \leq 1$, $|\na\psi_j|\leq 2^{j+3}/r$ in $B_{\bar{\rho}_j}$ and $\psi_j\equiv 1$ in $B_{j+1}$. Using the fractional Poincar\'e inequality to the function $w_j\psi_j$, we get 
 \begin{align}\label{eq19}
 	(b_j-b_{j+1})^p \Big( \frac{|B_{j+1} \cap \{ u <b_{j+1} \}|}{|B_{j+1}|} \Big)^{p/p^*_{s_1}} &\leq \Big(\Xint-_{B_{j+1}} w_j^{p^*_{s_1}} \psi_j^{p^*_{s_1}}dx \Big)^{p/p^*_{s_1}} \nonumber\\
 	&\leq c\Big(\Xint-_{B_{j}} w_j^{p^*_{s_1}}\psi_j^{p^*_{s_1}}dx \Big)^{p/p^*_{s_1}} \nonumber\\
 	&\leq c r^{s_1p} \Xint-_{B_j}\int_{B_j} \frac{|w_j(x)\psi_j(x) - w_j(y)\psi_j(y)|^p}{|x-y|^{N+s_1p}}dxdy. 
 \end{align}
Now, we estimate the right hand side quantity of \eqref{eq19} with the help of \eqref{eq18}. For the first term in \eqref{eq18}, we see that 
\begin{align}\label{eq20}
	&\int_{B_j}\int_{B_j} \frac{|\psi_j(x)-\psi_j(y)|^\el}{|x-y|^{N+s\el}} \big(w_j(x)^p+w_j(y)^p\big)dxdy \nonumber\\
	&\leq cb_j^\el \| \na\psi_j\|_{L^\infty}^\el \int_{B_j}\int_{B_j\cap\{u<b_j\}}  |x-y|^{\el-N-s\el}dxdy \nonumber\\
	&\leq c 2^{j\el} b_j^{\el}r^{-s\el}|B_j\cap\{u<b_j\}|.
\end{align}
For the second term, we observe that ${\rm supp}\psi_j \subset B_{\bar{\rho}_j}$. Therefore, for all $y\in \mb R^N\setminus B_j$, we have
 \begin{align*}
 	\sup_{x\in {\rm supp}\psi_j} |x-y|^{-N-s\el} \leq c2^{j(N+s\el)}|y-x_0|^{N+s\el}.
 	\end{align*}
 This implies that 
 \begin{align*}
 	 \sup_{x\in {\rm supp}\psi_j} \int_{\mb R^N\setminus B_j} \frac{(b_j+u(y)_-)^{\el-1}}{|x-y|^{N+s\el}}dy 
 	 &\leq c 2^{j(N+s\el)} \int_{\mb R^N\setminus B_j} \frac{(b_j+u(y)_-)^{\el-1}}{|x_0-y|^{N+s\el}}dy \\
 	 &\leq c 2^{j(N+s\el)}\left[b_j^{\el-1} r^{-s\el}+ \int_{\mb R^N\setminus B_R} \frac{(u(y)_-)^{\el-1}}{|x_0-y|^{N+s\el}}dy\right]  \\
 	 &= c 2^{j(N+s\el)} r^{-s\el} \left[ b_j^{\el-1}+ \Big(\frac{r}{R}\Big)^{s\el} T_{l-1}(u_-;x_0,R)^{\el-1}\right],
 \end{align*}
where we have used the fact that $u\geq 0$ in $B_R$. For $(\el,s)=(p,s_1)$, we have $b_j>\de k\geq \big(\frac{r}{R}\big)^{s_1p/(p-1)} T_{p-1}(u_-;x_0,R)$ (thanks to \eqref{eq14}), which yields
 \begin{align}\label{eq21}
 	\sup_{x\in {\rm supp}\psi_j} \int_{\mb R^N\setminus B_j} \frac{(b_j+u(y)_-)^{p-1}}{|x-y|^{N+s_1p}}dy \leq  c 2^{j(N+s_1p)} r^{-s_1p} b_j^{p-1}.
 \end{align}
For $(\el,s)=(q,s_2)$, on account of \eqref{eq14}, we have $b_j>\de k\geq \big(\frac{r}{R}\big)^{s_2q/(q-1)} T_{q-1}(u_-;x_0,R)$. Thus,
\begin{align}\label{eq22}
	\sup_{x\in {\rm supp}\psi_j} \int_{\mb R^N\setminus B_j} \frac{(b_j+u(y)_-)^{q-1}}{|x-y|^{N+s_2q}}dy \leq c 2^{j(N+s_2q)} r^{-s_2q} b_j^{q-1}.
\end{align}
For the third term in \eqref{eq18}, using the relation $b_j>\de\geq  r^\frac{\ga s_1p-N}{\ga(p-1)}\| f \|_{L^\ga(B_R)}^\frac{1}{p-1}$,  we obtain 
\begin{align}\label{eq22f}
	b_j \| f \|_{L^\ga(B_R)} |B_j\cap\{u<b_j\}|^\frac{1}{\ga'} \leq c b_j \| f \|_{L^\ga(B_R)} r^\frac{-N}{\ga} |B_j\cap\{u<b_j\}| \leq c b_j^p r^{-s_1p} |B_j\cap\{u<b_j\}|.
\end{align}
  Therefore, using \eqref{eq20}, \eqref{eq21}, \eqref{eq22}, \eqref{eq22f} and \eqref{eq18} in \eqref{eq19}, we deduce that
 \begin{align*}
 	(b_j-b_{j+1})^p \Big( \frac{|B_{j+1} \cap \{ u <b_{j+1} \}|}{|B_{j+1}|} \Big)^\frac{p}{p^*_{s_1}} &\leq c 2^{j(N+p+ps_1)}  \frac{|B_j\cap\{u<b_j\}|}{|B_j|} \big( b_j^{p}+b_j^{q}r^{s_1p-s_2q} \big) \nonumber\\
 	&\leq c 2^{j(N+p+ps_1)}  b_j^p \frac{|B_j\cap\{u<b_j\}|}{|B_j|},
 \end{align*}
 where in the last line we have used the fact that $b_j>\de k\geq r^\frac{s_1p-s_2q}{p-q}$ (consequently, $b_j^{p} \geq r^{s_1p-s_2q}b_j^q$). Setting 
 \begin{align*}
  A_j:= \frac{|B_{j} \cap \{ u <b_{j} \}|}{|B_{j}|},
 \end{align*}
 and using the relation $b_j - b_{j+1} \geq 2^{-j-3}b_j$, the above inequality implies that
 \begin{align*}
 	A_{j+1}^{\frac{p}{p^*_{s_1}}} \leq c 2^{j(N+p+ps_1)} A_j.
 \end{align*}
 Taking into account \eqref{eqlev0}, a standard iteration lemma (see for instance \cite[Lemma 2.6 and proof of Lemma 3.2, pp. 1823-1824]{dicastroHarn}) yields
 \begin{align*}
 	\lim_{j\to\infty} A_j=0,
 \end{align*}
that is, $\inf_{B_{4r}} u \ge \de k$, and noting \eqref{eq14}, we get the required result of the lemma. \QED
\end{proof}	

\begin{Lemma}\label{lem3}
  Let  $u\in X^{s_1,p}_{g}(\Om,\Om') \cap X^{s_2,q}_{g}(\Om,\Om')$ be a weak super-solution to problem \eqref{probMain}. Let $R\in (0,1)$ and $u\geq 0 $ in $B_R(x_0)\subset\Om$. Then, there exist constants $\tau\in (0,1)$ and $c\geq 1$ (both depending only on $N,s_1,p,s_2,q$) such that, for all $r\in (0,R)$, 
  \begin{align*}
  	\left(\Xint-_{B_r}u^\tau dx \right)^{1/\tau} \leq c\inf_{B_{r}} u +c r^\frac{s_1p-s_2q}{p-q} +c  r^\frac{\ga s_1p-N}{\ga(p-1)} \| f \|_{L^\ga(B_R)}^\frac{1}{p-1}+ c\sum_{(\el,s)}\Big(\frac{r}{R}\Big)^\frac{s\el}{\el-1}T_{\el-1}(u_-;x_0,R). 
  \end{align*}	
\end{Lemma}
\begin{proof}
 The proof of the lemma follows exactly on the similar lines of \cite[Proof of Lemma 4.1]{dicastroHarn} by taking $T:=\Big(\frac{r}{R}\Big)^\frac{s_1p}{p-1}T_{p-1}(u_-;x_0,R) +\Big(\frac{r}{R}\Big)^\frac{s_2q}{q-1}T_{q-1}(u_-;x_0,R)+ r^\frac{s_1p-s_2q}{p-q}+  r^\frac{\ga s_1p-N}{\ga(p-1)} \| f \|_{L^\ga(B_R)}^\frac{1}{p-1}$, there and using Lemma \ref{lem2} instead of \cite[Lemma 3.2]{dicastroHarn}. \QED
\end{proof}

\begin{Proposition}[Local boundedness]\label{localbdd}
 Let  $u\in X^{s_1,p}_{g}(\Om,\Om') \cap X^{s_2,q}_{g}(\Om,\Om')$ be a weak solution to problem \eqref{probMain}. Let $r\in (0,1)$ and $x_0\in\mb R^N$ be such that $B_r(x_0)\Subset\Om$. Then,   there exists a constant $C>0$ depending only on $N,s_1,p,s_2,q$ and $\ga$ (if $\ga<\infty$) such that, for all $\vep\in (0,1]$:
	 \begin{align*}
		\sup_{B_{r/2}} u_\pm &\leq  \vep T_{p-1}(u_\pm;x_0,\frac{r}{2})+ \vep T_{q-1}(u_\pm;x_0,\frac{r}{2})+ \vep r^\frac{s_1p-s_2q}{p-q} \nonumber\\
		&\quad+   \vep^\frac{p-\theta}{p} r^{\frac{\sg}{p\sg-1}(s_1p-N+\frac{N}{\sg})} \|f\|_{L^\ga(B_r)}^\frac{1}{p-1} + C\vep^\frac{(1-p)p^*_{s_1}}{p(p^*_{s_1}-p\sg)} \left(\Xint-_{B_{r}}u_\pm(x)^{p\sg}dx\right)^{1/(p\sg)},
	\end{align*}
 where  $\theta=\frac{(p^*_{s_1})'}{\ga}$ and $\sg=\frac{p-\theta}{p(1-\theta)}$.
\end{Proposition}
\begin{proof}
  We fix $\tl k\in\mb R^+$, $k\in\mb R$ and define the following: 
\begin{align*}
	r_i= (1+2^{-i})\frac{r}{2}, \ \ \tl r_i=\frac{r_i+r_{i+1}}{2};  	 \ \ k_i=k+(1-2^{-i})\tl k, \ \ \tl k_i= \frac{k_{i}+k_{i+1}}{2} \quad\mbox{for all }i\in\mb N.  
\end{align*}
Clearly, $r_{i+1}\leq \tl r_i \leq r_i$ and $k_i\leq \tl k_i$. Further,  we set
\begin{align*}
	&B_i= B_{r_i}(x_0), \ \ \tl B_{i}= B_{\tl r_i}(x_0); \quad \tl w_i:=  (u-\tl k_i)_+ \quad\mbox{and } w_i:= (u-k_i)_+,\\
	&  \quad \psi_i\in C^\infty_c(\tl B_i), \ 0\leq \psi_i \leq 1, \ \psi_i\equiv 1 \mbox{ in }B_{i+1}, \ \ |\na\psi_i| < 2^{i+3}/r. 
\end{align*}
Applying the fractional Poincar\'e inequality to the function $\tl w_i\psi_i$ (for the case $ps_1<N$, otherwise taking an arbitrarily large number in place of $p^*_{s_1})$, we have
\begin{align}\label{eq42}
	\left(  \Xint-_{B_i} |\tl w_i\psi_i|^{p^*_{s_1}} dx \right)^{p/p^*_{s_1}} 
	\leq c\frac{r^{s_1p}}{r^N}\int_{B_i}\int_{B_i} \frac{|\tl w_i(x)\psi_i(x)-\tl w_i(y)\psi_i(y)|^p}{|x-y|^{N+ps_1}}dxdy+   \Xint-_{B_i} |\tl w_i\psi_i|^pdx.
\end{align}
The first term on the right hand side is estimated by means of the following Caccioppoli type inequality (see \cite[Lemma 3.1]{JDS2}) 
\begin{align}\label{eq43}
 & \int_{B_i}\int_{B_i} \frac{|\tl w_i(x)\psi_i(x)-\tl w_i(y)\psi_i(y)|^p}{|x-y|^{N+ps_1}}dxdy \nonumber\\
 &\leq  C \sum_{(\el,s)} \int_{B_i}\int_{B_i} \frac{|\psi_i(x)-\psi_i(y)|^\el}{|x-y|^{N+s\el}}(\tl w_i(x)^\el+\tl w_i(y)^\el)dxdy \nonumber\\
 &\quad+ C   \sum_{(\el,s)} \Big(\sup_{y\in {\rm supp}\psi_i} \int_{\mb R^N\setminus B_i} \frac{\tl w_i(x)^{\el-1}dx}{|x-y|^{N+s\el}}\Big) \int_{B_i}\tl w_i\psi_i^p+  \int_{\Om} |f| \tl w_i\psi_i^p,
\end{align}
where $C=C(p,q)>0$ is a constant. Moreover, from the proof of \cite[Proposition 3.2]{JDS2}, we have 
  \begin{align}\label{eq44}
 	  \int_{B_i}\int_{B_i} \frac{|\psi_i(x)-\psi_i(y)|^\el}{|x-y|^{N+s\el}}(\tl w_i(x)^\el+ \tl w_i(y)^\el)dxdy 
 	  &\leq c \frac{2^{ip}}{\el(1-s)}\frac{1}{\tl k^{p-\el}}\frac{r^N}{r^{s\el}} \Xint-_{B_i}w_i^p(x)dx
 	\end{align}
and 
 \begin{align}\label{eq45}
 	\Big(\sup_{x\in {\rm supp}\psi_i} &\int_{\mb R^N\setminus B_i} \frac{\tl w_i(y)^{\el-1}dy}{|x-y|^{N+\el s}}\Big) \int_{B_i}\tl w_i(x)\psi_i^p(x)dx \nonumber\\
 	&\leq c\frac{2^{i(N+s\el+p-1)}}{\tl k^{p-1}} \frac{r^N}{r^{s\el}} T_{l-1}(w_0;x_0,r/2)^{\el-1} \Xint-_{B_i} w_i^p(x)dx.
 \end{align}
Now, we estimate the integral involving $f$. First, we note that if $\ga=\infty$, then 
 \begin{align*}
 	\int_{\Om} |f(x)| \tl w_i(x)\psi_i^p(x) dx \leq c \| f \|_{L^\infty(B_r)} r^N \Xint-_{B_i} w_i^p(x)dx.
 \end{align*}
Therefore, we assume that $\ga <\infty$ (consequently, $\ga'>1$). Applying H\"older's inequality and interpolation identity for $L^p$-spaces (with $1$ and $p^*_{s_1}$), we get 
\begin{align*}
	\int_{\Om}|f(x)|\tl w_i(x)\psi_i^p(x)dx &\leq  \| f \psi_i^{p-1} \|_{L^\ga(B_r)} \| \tl w_i\psi_i \|_{L^{\ga'}(B_i)} \\
	&\leq  \| f \|_{L^\ga(B_r)} \| \tl w_i\psi_i \|_{L^{p^*_{s_1}}(B_i)}^{\theta} \| \tl w_i\psi_i \|_{L^{1}(B_i)}^{1-\theta},
\end{align*}
where $\theta =\frac{(p^*_{s_1})'}{\ga} \in (0,1)$ (thanks to the assumption on $\ga$).
Using Young's inequality in the above expression (with the exponents $p/\theta$ and $p/(p-\theta)$), for $\e>0$, we deduce that
\begin{align}\label{eq47}
	\int_{\Om}|f(x)|\tl w_i(x)\psi_i^p(x)dx \leq \e  \| f \|_{L^\ga(B_r)}  \| \tl w_i\psi_i \|_{L^{p^*_{s_1}}(B_i)}^{p} + \e^{\frac{-\theta}{p-\theta}} \| f \|_{L^\ga(B_r)} \| \tl w_i\psi_i \|_{L^{1}(B_i)}^{\frac{p(1-\theta)}{p-\theta}}.
\end{align}
For convenience in writing, we denote $\sg=\frac{p-\theta}{p(1-\theta)}>1$, and observe that $p\sg\in (1,p^*_{s_1})$ (this is because of $\ga>N/(ps_1)$). Thus on account of the relation $\tl w_i\leq (\tl k_i - k_i)^{1-p\sg} w_i^{p\sg}$, we get 
 \begin{align*}
 	\Big(\int_{B_i} | \tl w_i\psi_i |dx\Big)^{1/\sg} \leq \Big( \frac{1}{(\tl k_i - k_i)^{p\sg-1}} \int_{B_i}  (\tl w_i)^{p\sg}dx\Big)^{1/\sg} \leq \frac{1}{\tl k^{p-1/\sg}} \Big(\int_{B_i} (\tl w_i)^{p\sg}dx\Big)^{1/\sg}.
 \end{align*}
Taking $\e= \frac{1}{2cC \| f \|_{L^\ga(B_r)}}$ in the above expression, where $c$ and $C$ are as in \eqref{eq42} and \eqref{eq43}, respectively, from \eqref{eq47}, we get 
\begin{align}\label{eq48}
	\int_{\Om}|f(x)|\tl w_i(x)\psi_i^p(x)dx \leq \frac{1}{2cC} r^{N-s_1p} \Big(\Xint-_{B_i} |\tl w_i\psi_i |^{p^*_{s_1}}\Big)^\frac{p}{p^*_{s_1}} + C_1  \frac{\| f \|_{L^\ga(B_r)}^\frac{p}{p-\theta} }{\tl k^{p-1/\sg}} r^\frac{N}{\sg} \Big(\Xint-_{B_i} (\tl w_i)^{p\sg}dx\Big)^\frac{1}{\sg}, 
\end{align} 
where $C_1=C_1(c,C,N,\sg)>0$ is a constant. Therefore, combining \eqref{eq44}, \eqref{eq45} and \eqref{eq48} with \eqref{eq43}, from \eqref{eq42}, we obtain  
  \begin{align*}
	\left(\Xint-_{B_i} |\tl w_i\psi_i|^{p^*_{s_1}}dx\right)^\frac{p}{p^*_{s_1}} \leq & c 2^{i(N+s_1p+p-1)} \left[ \frac{1}{p(1-s_1)}+ \frac{\sum_{(\el,s)} r^{s_1p-s\el}T_{\el-1}(w_0;x_0,\frac{r}{2})^{\el-1}}{\tl k^{p-1}} \right.\nonumber\\
	&\left. +\frac{r^{s_1p-s_2q}}{q(1-s_2)\tl k^{p-q}}+r^{s_1p-N+\frac{N}{\sg}} \frac{\| f \|_{L^\ga(B_r)}^\frac{p}{p-\theta} }{\tl k^\frac{p(p-1)}{p-\theta}}  +1 \right] 
    \Big(\Xint-_{B_i} (w_i)^{p\sg}dx\Big)^\frac{1}{\sg},
	\end{align*}
 where we have used H\"older's inequality in the integrals involving $w_i^p$.	
 For $\vep\in (0,1]$, we take 
  \begin{align}\label{eq49}
		\tl k \geq  \vep T_{p-1}(w_0;x_0,\frac{r}{2})+ \vep T_{q-1}(w_0;x_0,\frac{r}{2})+ \vep r^\frac{s_1p-s_2q}{p-q}+  \|f\|_{L^\ga(B_r)}^\frac{1}{p-1} \vep^\frac{p-\theta}{p} r^{\frac{p-\theta}{p(p-1)}(s_1p-N+\frac{N}{\sg})},
	\end{align}
 (note that $s_1p-N+\frac{N}{\sg}>0$). Thus, the above expression yields
 \begin{align}\label{eq50}
 	\left(\Xint-_{B_i} |\tl w_i\psi_i|^{p^*_{s_1}}dx\right)^\frac{p}{p^*_{s_1}} \leq & c 2^{i(N+s_1p+p-1)} \frac{1}{\vep^{p-1}} \Big(\Xint-_{B_i} (w_i)^{p\sg} dx\Big)^\frac{1}{\sg},
 \end{align}
 where we have used the relation
	 $\big(r^{s_1p-s_2q} \tl k^{q-p}\big) \big(\tl k^{1-q} T_{q-1}(w_0;x_0.r/2)^{q-1} \big) \leq \vep^{1-p}$. 
Now, we estimate the left hand side term as \cite[Proof of Theorem 1.1, p. 1292]{dicastro}:
\begin{align*}
	\left(\Xint-_{B_i} |\tl w_i\psi_i|^{p^*_{s_1}}dx\right)^\frac{p}{p^*_{s_1}} &\geq (k_{i+1}-\tl k_i)^\frac{(p^*_{s_1}-p\sg)p}{p^*_{s_1}} \Big(\Xint-_{B_{i+1}}  (w_{i+1})^{p\sg}dx\Big)^{p/p^*_{s_1}} \\
	&= \Big(\frac{\tl k}{2^{i+2}}\Big)^\frac{(p^*_{s_1}-p\sg)p}{p^*_{s_1}} \Big(\Xint-_{B_{i+1}}  (w_{i+1})^{p\sg}dx\Big)^{p/p^*_{s_1}}.
\end{align*}
 Then, setting $A_i:=\big(\Xint-_{B_i} w_i^{p\sg}(x)dx\big)^{1/(p\sg)}$, from \eqref{eq50}, we have
 \begin{align*}
\tl k^\frac{(p^*_{s_1}-p\sg)p}{p^*_{s_1}} A_{i+1}^\frac{p^2\sg}{p^*_{s_1}} \leq c 2^{i\big(N+s_1p+p-1+\frac{(p^*_{s_1}-p\sg)p}{p^*_{s_1}}\big)} \vep^{1-p} A_i^p,	
 \end{align*}
 that is,
 \begin{align*}
 	\frac{A_{i+1}}{\tl k}  \leq c \vep^{\frac{p^*_{s_1}(1-p)}{p^2\sg}} C_2^{i} \left(\frac{A_{i}}{\tl k} \right)^{1+\vartheta},
 \end{align*}
 where $\vartheta=\frac{p^*_{s_1}}{p\sg}-1>0$ and $C_2 := 2^{\big(\frac{N+s_1p+p-1}{p}+\frac{p^*_{s_1}-p\sg}{p^*_{s_1}}\big)\frac{p^*_{s_1}}{p\sg}}>1$. We note that $\frac{p^*_{s_1}(p-1)}{p^2\sg\vartheta}= \frac{p-1}{p}\frac{p^*_{s_1}}{p^*_{s_1}-p\sg}$ and $\frac{p-\theta}{p(p-1)}=\frac{\sg}{p\sg-1}$. Next, we choose
 \begin{align*}
 	\tl k &= \vep T_{p-1}(w_0;x_0,\frac{r}{2})+ \vep T_{q-1}(w_0;x_0,\frac{r}{2})+ \vep r^\frac{s_1p-s_2q}{p-q}+  \|f\|_{L^\ga(B_r)}^\frac{1}{p-1} \vep^\frac{p-\theta}{p} r^{\frac{\sg}{p\sg-1}(s_1p-N+\frac{N}{\sg})} \nonumber\\
 	&\quad+ \vep^\frac{(1-p)p^*_{s_1}}{p(p^*_{s_1}-p\sg)} c^{1/\vartheta} C_2^{1/\vartheta^2} A_0,
 \end{align*}
which clearly satisfies \eqref{eq49} and 
\begin{align*}
\frac{A_{0}}{\tl k} \leq \vep^{\frac{p^*_{s_1}(p-1)}{p^2\sg\vartheta}} c^{-1/\vartheta} C_2^{-1/\vartheta^2}.  	
\end{align*}
Therefore, a well known iteration argument implies that
 \begin{align*}
		A_i \to 0 \quad\mbox{as }i\to\infty.
	\end{align*}
 Thus, we obtain
 \begin{align*}
 	 \sup_{B_{r/2}} (u-k)_+ \leq \tl k& = \vep T_{p-1}((u-k)_+;x_0,\frac{r}{2})+ \vep T_{q-1}((u-k)_+;x_0,\frac{r}{2})+ \vep r^\frac{s_1p-s_2q}{p-q} \nonumber\\
 	 &+   \vep^\frac{p-\theta}{p} r^{\frac{\sg}{p\sg-1}(s_1p-N+\frac{N}{\sg})} \|f\|_{L^\ga(B_r)}^\frac{1}{p-1} + C\vep^\frac{(1-p)p^*_{s_1}}{p(p^*_{s_1}-p\sg)} \left(\Xint-_{B_{r}}(u-k)_+^{p\sg}\right)^{1/(p\sg)}.
 \end{align*}
 Now, the proof of the proposition, for $u_+$, follows by taking $k=0$. The proof for the case $u_-$ runs analogously.\QED
\end{proof}
By slightly modifying the choice of $\tl k$ in the above proof, we have:
\begin{Corollary}\label{corlocalbd}
	Under the hypothesis of Proposition \ref{localbdd}, the following holds: 
	\begin{align*}
		\sup_{B_{r/2}} u_+ \leq & \vep T_{p-1}(u_+;x_0,\frac{r}{2})+ \vep r^\frac{s_1p-s_2q}{p-1} T_{q-1}(u_+;x_0,\frac{r}{2})^\frac{q-1}{p-1}+ \vep r^\frac{s_1p-s_2q}{p-q} \nonumber\\
		&+   \vep^\frac{p-\theta}{p} r^{\frac{\sg}{p\sg-1}(s_1p-N+\frac{N}{\sg})} \|f\|_{L^\ga(B_r)}^\frac{1}{p-1} + C\vep^\frac{(1-p)p^*_{s_1}}{p(p^*_{s_1}-p\sg)} \left(\Xint-_{B_{r}}u_+^{p\sg}dx\right)^{1/(p\sg)},
	\end{align*}
	where  $C=C(N,p,q,s_1,s_2,\ga)>0$ is a constant.
\end{Corollary}

Next, we prove the following version of Caccioppoli type inequality.
\begin{Lemma}\label{lem4}
 Let $m\in (1,p)$ and $\la>0$. Let  $u\in X^{s_1,p}_{g}(\Om,\Om') \cap X^{s_2,q}_{g}(\Om,\Om')$ be a weak super-solution to problem \eqref{probMain} such that $u\geq 0 $ in $B_R(x_0)\subset\Om$, for $R\in (0,1)$. Then, for any $r\in (0,3R/4)$ and $\phi\in C_c^\infty(B_r)$, with $0\leq \phi\leq 1$, there holds
\begin{align*}
 	&\int_{B_r} \int_{B_r} \frac{|w_p(x)\phi(x)-w_p(y)\phi(y)|^p}{|x-y|^{N+s_1p}}dxdy\\
 &\leq  \int_{\Om}|f(x)|\bar u^{1-m}(x)\phi^p(x)dx +
 c \sum_{(\el,s)}\la^{\el-p} \int_{ B_{r}}\int_{ B_{r}}\frac{|\phi(x)-\phi(y)|^\el}{|x-y|^{N+s\el}}(w_p(x)^p+w_p(y)^p)dxdy \nonumber\\
 &\quad+ c \sum_{(\el,s)}\Big( \sup_{x\in {\rm supp}\phi} \int_{{\mb R^N}\setminus B_{r}} \frac{dy}{|x-y|^{N+s\el}} + \la^{1-\el} \int_{{\mb R^N}\setminus B_{R}} \frac{u(y)^{\el-1}_-dy}{|y-x_0|^{N+s\el}}\Big) 
\int_{B_r} w_\el(x)^\el\phi^p(x)dx ,
 \end{align*}
where $w_\el:=(u+\la)^\frac{\el-m}{\el} $, and $c=c(p,q,m)>0$ is a constant.
\end{Lemma}
\begin{proof}
 Set $\bar{u}:= u+\la$ and let $\eta:= \bar{u}^{1-m}\phi^p$, where $m\in [1+\e,p-\e]$, for $\e>0$ small. Similarly to Lemma \ref{lem1}, we note that $\bar u$ is also a weak super-solution, therefore taking $\eta$ as a test function, we get  
  \begin{align}\label{eq23}
 	0 &\leq \sum_{(\el,s)}\int_{B_{r}} \int_{B_{r}}  \frac{[\bar u(x)-\bar u(y)]^{\el-1}}{|x-y|^{N+s\el}}(\eta(x)-\eta(y))dxdy \nonumber\\
 	&\quad +2 \sum_{(\el,s)}\int_{{\mb R^N}\setminus B_{r}}\int_{B_{r}} \frac{[\bar u(x)-\bar u(y)]^{\el-1}}{|x-y|^{N+s\el}}\eta(x)dxdy
 		-\int_{\Om}f(x)\eta(x)dx \nonumber \\
 	&=:I_1(\el)+I_2(\el)-I_3(f). 
 \end{align}
 We observe that, for all $x\in B_r$  and $y\in \mb R^N$ (note that $u\ge 0$ in $B_R$), 
 \begin{align*}
 	[\bar u(x)-\bar u(y)]^{\el-1} \leq c \big(\bar{u}(x)^{\el-1}+u(y)^{\el-1}_- \big) \quad\mbox{and }
 	\bar u^{1-m}(x) \leq \la^{1-\el}\bar u^{\el-m}(x).
 \end{align*}
 Therefore, setting $d\mu_\el=|x-y|^{-N-s\el}dxdy$, we deduce that
 \begin{align}\label{eq24}
  I_2(\el) &\leq c \int_{{\mb R^N}\setminus B_{r}}\int_{B_{r}} \bar u^{\el-1}(x)\bar{u}^{1-m}(x)\phi^p(x) d\mu_\el+ c\int_{\mb R^N\setminus B_{r}}\int_{ B_{r}} u(y)^{\el-1}_-\bar{u}^{1-m}(x)\phi^p(x) d\mu_\el \nonumber\\
  &\leq c \Big( \sup_{x\in {\rm supp}\phi} \int_{\mb R^N\setminus B_{r}} |x-y|^{-N-s\el}dy + \la^{1-\el} \int_{\mb R^N\setminus B_{R}} \frac{u(y)^{\el-1}_-}{|y-x_0|^{N+s\el}}dy\Big) \int_{B_r} w_\el(x)^\el\phi^p(x)dx,
 \end{align}
 where we have used $u\geq 0$ in $B_R$ and $w_\el:=
 	\bar u^\frac{\el-m}{\el}$. Next, for the first integral $I_1(p)$, from \cite[proof of Lemma 5.1, pp.1830-1833]{dicastroHarn}, we have 
\begin{align}\label{eq25}
 I_1(p)\leq -c \int_{ B_{r}}\int_{ B_{r}} |w_p(x)-w_p(y)|^p\phi^p(y) d\mu_p+ c\int_{ B_{r}}\int_{ B_{r}}|\phi(x)-\phi(y)|^p(w_p(x)^p+w_p(y)^p)d\mu_p.
\end{align}
 To estimate $I_1(q)$, without loss of generality, we assume $\bar u(x)>\bar u(y)$. We proceed exactly as in Step II of Lemma \ref{lem1} to get
  {\small \begin{align*}
 	(\bar u(x)-\bar u(y))^{q-1} \left[ \frac{\phi^p(x)}{\bar u(x)^{m-1}} - \frac{\phi^p(y)}{\bar u(y)^{m-1}}\right] 
 	&\leq \phi^p(y) \frac{(\bar u(x)-\bar u(y))^{q-1}}{\bar u(x)^{m-1}} \left[ 1+c\de \frac{\bar u(x)-\bar u(y)}{\bar u(x)}   -\frac{\bar u(x)^{m-1}}{\bar u(y)^{m-1}} \right] \\
 	&\quad+c\de^{1-q} \bar u(x)^{q-m} |\phi(x)-\phi(y)|^q.
 \end{align*}}
For suitable choice of $\de\in (0,1)$, from \cite[Proof of Lemma 5.1, p.1831]{dicastroHarn} (see the expression for $g(t)$ there), we have 
 \begin{align*}
 	\left[ 1+c\de \frac{\bar u(x)-\bar u(y)}{\bar u(x)}   -\frac{\bar u(x)^{m-1}}{\bar u(y)^{m-1}} \right] \leq 0,
 \end{align*}
 and an analogous result holds for the case $\bar u(x)< \bar u(y)$. Thus, 
 \begin{align}\label{eq26}
 I_1(q) \leq c \la^{q-p}\int_{ B_{r}}\int_{ B_{r}} 	\bar u(x)^{p-m} |\phi(x)-\phi(y)|^q d\mu_q.
 \end{align}
Therefore, combining \eqref{eq24}, \eqref{eq25}, \eqref{eq26} and using \eqref{eq23}, we obtain
 \begin{align*}
	&\int_{ B_{r}}\int_{ B_{r}} |w_p(x)-w_p(y)|^p\phi^p(y) d\mu_p \\
	&\leq  \int_{\Om}|f(x)|\bar u^{1-m}(x)\phi^p(x)dx+
	c \sum_{(\el,s)}\la^{\el-p} \int_{ B_{r}}\int_{ B_{r}} |\phi(x)-\phi(y)|^\el(w_p(x)^p+w_p(y)^p)d\mu_\el \nonumber\\
	&+ c \sum_{(\el,s)}\Big( \sup_{x\in {\rm supp}\phi} \int_{\mb R^N\setminus B_{r}} \frac{dy}{|x-y|^{N+s\el}} + \la^{1-\el} \int_{\mb R^N\setminus B_{R}} \frac{u(y)^{\el-1}_-dy}{|y-x_0|^{N+s\el}}\Big) \int_{B_r} w_\el(x)^\el\phi^p(x)dx.
\end{align*}
 By observing 
 \begin{align*}
 	|w_p(x)\phi(x)-w_p(y)\phi(y)|^p \leq c (w_p(x)^p+w_p(y)^p) |\phi(x)-\phi(y)|^p + c|w_p(x)-w_p(y)|^p \phi(y)^p,
 \end{align*}
we complete the proof of the lemma. \QED
\end{proof}

\begin{Lemma}\label{lem5}
	Let  $u\in X^{s_1,p}_{g}(\Om,\Om') \cap X^{s_2,q}_{g}(\Om,\Om')$ be a solution to problem \eqref{probMain} such that $u\geq 0$ in $B_R(x_0)\Subset\Om$, for some $R\in (0,1)$. Then,  for all $0<r<R$, the following holds:
	\begin{align*}
		&T_{p-1}(u_+;x_0,r)^{p-1}+ r^{s_1p-s_2q} T_{q-1}(u_+;x_0,r)^{q-1} \\ &\leq C (\sup_{B_r} u\big)^{p-1}+ C\Big(\frac{r}{R}\Big)^{s_1p} T_{p-1}(u_-;x_0,R)^{p-1}  + C r^{s_1p-s_2q} \Big(\frac{r}{R}\Big)^{s_2q} T_{q-1}(u_-;x_0,R)^{q-1}\\
		&\quad+ C r^\frac{(s_1p-s_2q)(p-1)}{p-q} +C \|f\|_{L^\ga(B_R)}r^\frac{\ga s_1p-N}{\ga(p-1)}, 
	\end{align*}
	where  $C=C(N,p,q,s_1,s_2)>0$ is a constant.
\end{Lemma}
\begin{proof}
 Let $k:=\sup_{B_r} u$ and $\phi\in C_c^\infty(B_r)$ such that $0\leq \phi \leq 1$ and $|\na \phi|\leq 8/r$ in $B_r$, and $\phi\equiv 1$ in $B_{r/2}$. Then, testing the equation with $\eta:= (u-2k)\phi$ (note that $\eta\leq 0$), we get 
	\begin{align}\label{eq40}
		0 &= \sum_{(\el,s)} \int_{B_{r}} \int_{B_{r}}  \frac{[u(x)- u(y)]^{\el-1}}{|x-y|^{N+s\el}}(\eta(x)-\eta(y))dxdy \nonumber \\
		&\quad+2\sum_{(\el,s)} \int_{{\mb R^N}\setminus B_{r}}\int_{B_{r}} \frac{[u(x)- u(y)]^{\el-1}}{|x-y|^{N+s\el}}\eta(x) dxdy
		-\int_{\Om}f(x)\eta(x)dx \nonumber \\
		&=:I_1(\el)+I_2(\el)-I_3(f). 
	\end{align}
	Proceeding similarly to \cite[Lemma 4.2]{dicastroHarn},  we have 
	\begin{align*}
		&I_2(\el) \geq -c k^{\el}r^{-s\el}|B_r|-ck|B_r|R^{-s\el}T_{\el-1}(u_-;x_0,R)^{\el-1}+ck|B_r| r^{-s\el} T_{\el-1}(u_+;x_0,r)^{\el-1} \mbox{ and} \\
		& I_1(p) \geq 	-c k^{p}r^{-s_1p}|B_r|.
	\end{align*}	
 To estimate $I_1(q)$, following the proof of \cite[Lemma 3.1]{JDS2}, with the notation $\tl u=(u-2k)$, we get 
	\begin{align*}
		[\tl u(x)- \tl u(y)]^{q-1} (\tl u(x)\phi(x)-\tl u(y)\phi(y)) &\geq \frac{1}{2} |\tl u(x)-\tl u(y)|^p \phi(x)^p - c|\tl u(y)|^q |\phi(x)-\phi(y)|^q \\
		&\geq -ck^{q}|\phi(x)-\phi(y)|^q,
	\end{align*}
	where in the last line we have used $|\tl u|\leq 3k$ in $B_r$. Thus, 
	\begin{align*}
		I_1(q) \geq - ck^q \int_{ B_{r}}\int_{ B_{r}} \frac{|\phi(x)-\phi(y)|^q}{|x-y|^{N+s_2q}}dxdy
		&\geq r^{-s_2q}|B_r|- ck^q \int_{ B_{r}}\int_{ B_{r}} |x-y|^{q-N-s_2q}dxdy\\
		&\geq -ck^{q}r^{-s_2q}|B_r|.
	\end{align*}
 For $I_3(f)$, noting $|u-k|\leq 3k$ in $B_r$, we have 
 \begin{align*}
	I_3(f) \leq 3k \int_{ B_{r}} |f|\phi dx \leq 3 k \| f \|_{L^\ga(B_r)} |B_r|^{1/\ga'}.
 \end{align*}
 Collecting all informations in \eqref{eq40},  we obtain (upon multiplication with $r^{s_1p}(k|B_r|)^{-1}$)
	\begin{align*}
	 &T_{p-1}(u_+;x_0,r)^{p-1}+ r^{s_1p-s_2q} T_{q-1}(u_+;x_0,r)^{q-1} \\ &\leq C\left[ k^{p-1}+   r^{s_1p-s_2q} k^{q-1}+ \|f\|_{L^\ga(B_r)}r^\frac{\ga s_1p-N}{\ga(p-1)} \right] +C\Big(\frac{r}{R}\Big)^{s_1p} T_{p-1}(u_-;x_0,R)^{p-1} \\
	 &\quad+ C r^{s_1p-s_2q} \Big(\frac{r}{R}\Big)^{s_2q} T_{q-1}(u_-;x_0,R)^{q-1}.
	\end{align*}
 This proves the lemma on account of the relation $cr^{s_1p-s_2q}k^{q-1} \leq ck^{p-1}+c r^\frac{(s_1p-s_2q)(p-1)}{p-q}$ (by Young's inequality). \QED 
\end{proof}

\section{Harnack and weak Harnack inequalities}
 In this section, we prove the Harnack and weak Harnack type inequalities for fractional $(p,q)$-problems. We start with the following result.
 \begin{Proposition}[Weak Harnack inequality]\label{weakharn}
  Let  $u\in X^{s_1,p}_{g}(\Om,\Om') \cap X^{s_2,q}_{g}(\Om,\Om')$ be a weak super-solution to problem \eqref{probMain} such that $u\geq 0 $ in $B_R(x_0)\subset\Om$, for $R\in (0,1)$. Then, for any $r\in (0,R)$ and for any $t<\frac{N(p-1)}{N-s_1p}$ (for $ps_1<N$), there holds:
 	\begin{align*}
 	  \left(\Xint-_{B_{r/2}} u^t dx \right)^{1/t} \leq &c \inf_{B_{r}}+ c\Big(\frac{r}{R}\Big)^\frac{s_1p}{p-1}T_{p-1}(u_-;x_0,R) +c\Big(\frac{r}{R}\Big)^\frac{s_2q}{q-1}T_{q-1}(u_-;x_0,R)+c r^\frac{s_1p-s_2q}{p-q} \\
 	  &+c  r^\frac{\ga s_1p-N}{\ga(p-1)}  \| f \|_{L^\ga(B_R)}^\frac{1}{p-1},
 	\end{align*}
 	where $c=c(N,p,s_1,q,s_2)>0$ is a constant.
 \end{Proposition}
 \begin{proof}
 Let $1/2 < \kappa'<\kappa \leq 3/4$ and let $\psi\in C_c^\infty(B_{\kappa r})$ be such that $\phi\equiv 1$ in $B_{\kappa'r}$, $0\leq \phi\leq 1$ and $|\na \phi|\leq \frac{4}{(\kappa-\kappa')r}$. Set $w\equiv w_p=\bar u^\frac{p-m}{p}:=(u+\la)^\frac{p-m}{p}$, for $\la>0$. Applying the fractional Poincar\'e inequality to the function $w\phi$,  we get
 	\begin{align}\label{eq27}
 		\left(\Xint-_{B_r} |w(x)\phi(x)|^{p^*_{s_1}}dx\right)^\frac{p}{p^*_{s_1}} \leq c r^{s_1p} \Xint-_{B_r}\int_{B_r} \frac{|w(x)\phi(x)-w(y)\phi(y)|^p}{|x-y|^{N+s_1p}}dxdy.
 	\end{align}  
 To estimate the right hand side quantity, we will use Lemma \ref{lem4}. Using the bound on $|\na\phi|$, we have 
 	\begin{align*}
 		\int_{ B_{r}}\int_{ B_{r}} \frac{|\phi(x)-\phi(y)|^\el}{|x-y|^{N+s\el}}(w_p(x)^p+w_p(y)^p)dxdy \leq \frac{cr^{-s\el}}{(\kappa-\kappa')^\el} \int_{ B_{\kappa r}} w_p(x)^p dx.	
 	\end{align*}
 	Moreover, using ${\rm supp}\phi\subset B_{\kappa r}$, we see that 
 	\begin{align*}
 		\sup_{x\in {\rm supp}\phi} \int_{\mb R^N\setminus B_r} |x-y|^{-N-s\el}dy \leq c r^{-s\el}.
 	\end{align*} 
 Furthermore, employing H\"older's inequality and interpolation result for $L^p$-spaces (observe that $\bar u\geq \la$ and $p<p\ga'<p^*_{s_1}$), we deduce that 
 \begin{align*}
 	\int_{\Om}|f|\bar u^{1-m}\phi^pdx \leq \la^{1-p} \int_{\Om}|f|\bar u^{p-m} \phi^p dx&\leq \la^{1-p} \| f \|_{L^\ga(B_R)} \| w\phi \|_{L^{p\ga'}(B_R)}^p. 
 	\\
 	&\leq \la^{1-p} \| f \|_{L^\ga(B_R)} \| w\phi \|_{L^{p^*_{s_1}}(B_R)}^\frac{N}{\ga s_1} \| w\phi \|_{L^{p}(B_R)}^\frac{\ga s_1p-N}{\ga s_1},
 \end{align*}
 which on using Young's inequality (with the exponents $ \frac{\ga s_1p}{N}$ and $\frac{\ga s_1p}{\ga s_1p-N}$), for $\e>0$, yields
 	 \begin{align*}
 	 	\int_{B_r}|f(x)|\bar u^{1-m}(x)\phi^p(x)dx &\leq C_1\e \la^{1-p} \| f \|_{L^\ga(B_R)}  r^{N-s_1p} 	\left(\Xint-_{B_r} |w\phi|^{p^*_{s_1}} dx\right)^{p/p^*_{s_1}} \\
 	 	&\quad+ C_1\e^{-\frac{N}{\ga s_1p-N}} \la^{1-p} \| f \|_{L^\ga(B_R)} r^N \Xint-_{B_r}|w\phi|^p dx, 
 	 \end{align*}
  where $C_1>0$ is a constant which depends only on $N,s_1,p$. Choosing $\e=\frac{\la^{p-1}}{2c C_1}\frac{1}{\| f \|_{L^\ga(B_R)}}$, with $c$ as in \eqref{eq27}, we get
 	\begin{align*}
 		\Xint-_{B_r}|f(x)|\bar u^{1-m}(x)\phi^p(x)dx \leq \frac{r^{-s_1p}}{2c} 	\left(\Xint-_{B_r} |w\phi|^{p^*_{s_1}}\right)^\frac{p}{p^*_{s_1}} 
 		+ c_2  \Big(\frac{\| f \|_{L^\ga(B_R)}}{\la^{p-1}}\Big)^\frac{\ga s_1p}{\ga s_1p-N} \Xint-_{B_r}|w\phi|^p dx,
 	\end{align*}
 where $c_2=c_2(N,s_1,p,c)>0$ is a constant. Collecting all these informations in \eqref{eq27} and using Lemma \ref{lem4}, we obtain
 	\begin{align*}
 		\left(\Xint-_{B_r} |w\phi|^{p^*_{s_1}}dx\right)^\frac{p}{p^*_{s_1}} &\leq c r^{s_1p} \left[\frac{r^{-s_1p}}{(\kappa-\kappa')^p}+  \frac{\la^{q-p}r^{-s_2q}}{(\kappa-\kappa')^q}+ \Big(\frac{\| f \|_{L^\ga(B_R)}}{\la^{p-1}}\Big)^\frac{\ga s_1p}{\ga s_1p-N} \right]\Xint-_{ B_{\kappa r}} w(x)^p dx \\
 		& \ + c r^{s_1p} \left[r^{-s_1p}+ \la^{1-p} R^{-s_1p} T_{p-1}(u_-;x_0,R)^{p-1}  \right]\Xint-_{B_{\kappa r}}w_p(x)^pdx \\
 		& \ +c  r^{s_1p} \left[r^{-s_2p}+ \la^{1-p} R^{-s_2p} T_{q-1}(u_-;x_0,R)^{q-1}  \right]\Xint-_{B_{\kappa r}}w_q(x)^qdx. 
 	\end{align*}
 	Noting the fact that $w_q(x)^q=\bar u(x)^{q-m} \leq \la^{q-p}\bar u(x)^{p-m}= \la^{q-p}w_p(x)^p$, we have 
 	{\small\begin{align}\label{eq28}
 			\left(\Xint-_{B_r} |w(x)\phi(x)|^{p^*_{s_1}}dx\right)^\frac{p}{p^*_{s_1}} &\leq \frac{c}{(\kappa-\kappa')^p} \left[1+\la^{q-p}r^{s_1p-s_2q}+\la^{1-p} \Big(\frac{r}{R}\Big)^{s_1p} T_{p-1}(u_-;x_0,R)^{p-1} \right. \nonumber\\
 			&\left. \quad+ \la^{1-p} r^{s_1p-s_2q} \Big(\frac{r}{R}\Big)^{s_2q}T_{q-1}(u_-;x_0,R)^{q-1}+r^{s_1p} \Big(\frac{\| f \|_{L^\ga(B_R)}}{\la^{p-1}}\Big)^\frac{\ga s_1p}{\ga s_1p-N}\right] \nonumber\\
 			&\qquad \times \Xint-_{B_{\kappa r}}w_p(x)^pdx.
 	\end{align}}
 	We take 
 	\begin{align}\label{eq30}
 		\la=\Big(\frac{r}{R}\Big)^\frac{s_1p}{p-1}T_{p-1}(u_-;x_0,R) +\Big(\frac{r}{R}\Big)^\frac{s_2q}{q-1}T_{q-1}(u_-;x_0,R)+ r^\frac{s_1p-s_2q}{p-q}+ \| f \|_{L^\ga(B_R)}^\frac{1}{p-1} r^\frac{\ga s_1p-N}{\ga(p-1)}.
 	\end{align}
 Moreover,  as in the proof of Lemma \ref{lem1},  observe that
 	\begin{align*}
 		\la^{1-p}& r^{s_1p-s_2q}\Big(\frac{r}{R}\Big)^{s_2q}T_{q-1}(u_-;x_0,R)^{q-1} \\
 		&= \la^{1-q} \Big(\frac{r}{R}\Big)^{s_2q}T_{q-1}(u_-;x_0,R)^{q-1} \la^{q-p} r^{s_1p-s_2q} 
 		\leq 1
 	\end{align*}
 	and 
 	 \begin{align*}
 	 	\Big(\frac{\| f \|_{L^\ga(B_R)}}{\la^{p-1}}\Big)^\frac{\ga s_1p}{\ga s_1p-N} \leq r^{-\frac{\ga s_1 p-N}{\ga}\times \frac{\ga s_1 p}{\ga s_1p-N}} = r^{-s_1p}.
 	 \end{align*}
 	Therefore, from \eqref{eq28} and noting that $\phi\equiv 1$ in $B_{\kappa' r}$, we get 
 	\begin{align*}
 		\left(\Xint-_{B_{\kappa'r}} |w(x)|^{p^*_{s_1}}dx\right)^{p/p^*_{s_1}} \leq c\left(\Xint-_{B_r} |w(x)\phi(x)|^{p^*_{s_1}}dx\right)^{p/p^*_{s_1}} \leq \frac{c}{(\kappa-\kappa')^p} \Xint-_{B_{\kappa r}}w(x)^pdx,
 	\end{align*} 
 	that is,
 	\begin{align*}
 		\left(\Xint-_{B_{\kappa'r}} \bar u(x)^{\frac{(p-m)N}{N-s_1p}}dx\right)^\frac{N-s_1p}{N} \leq \frac{c}{(\kappa-\kappa')^p} \Xint-_{B_{\kappa r}}\bar u(x)^{p-m}dx.
 	\end{align*}
 	Thus, proceeding as in \cite[(5.14), p.1834]{dicastroHarn} (or one can use a standard finite Moser iteration argument similar to  \cite[Theorem 8.18]{gilbarg}), we have 
 	\begin{align*}
 		\left(\Xint-_{B_{r/2}} \bar u(x)^tdx\right)^{1/t} \leq c \left(\Xint-_{B_{3r/4}} \bar u(x)^{t_1}dx\right)^{1/t_1} \quad\mbox{for all }0<t_1<t<\frac{N(p-1)}{N-s_1p}.
 	\end{align*}
 	Applying Lemma \ref{lem3}, for $t_1=\tau$, we obtain 
 	\begin{align}\label{eq29}
 		\left(\Xint-_{B_{r/2}} \bar u(x)^tdx\right)^{1/t} \leq& c\inf_{B_{r}} \bar u + c\Big(\frac{r}{R}\Big)^\frac{s_1p}{p-1}T_{p-1}(u_-;x_0,R) +c\Big(\frac{r}{R}\Big)^\frac{s_2q}{q-1}T_{q-1}(u_-;x_0,R) \nonumber\\
 		& +c r^\frac{s_1p-s_2q}{p-q}+c  \| f \|_{L^\ga(B_R)}^\frac{1}{p-1} r^\frac{\ga s_1p-N}{\ga(p-1)}. 
 	\end{align}
 	Then, noticing $\bar u=u+\la$, we observe that 
 	\begin{align*}
 		\left(\Xint-_{B_{r/2}} u(x)^tdx\right)^{1/t} \leq \left(\Xint-_{B_{r/2}} \bar u(x)^tdx\right)^{1/t}.
 	\end{align*}
 	This combined with \eqref{eq29} and the definition of $\la$ (given by \eqref{eq30}) proves the proposition. \QED
 \end{proof}
 By replacing the constant $c$ appearing in Proposition \ref{weakharn} with $\bar c=\max\{c,2\}$, we obtain the following result.
 \begin{Corollary}\label{corweakharn}
 	Under the hypothesis of Proposition \ref{weakharn}, for all $0<t<\frac{N(p-1)}{N-ps_1}$, we have 
 	\begin{align*}
 		\inf_{B_{r}} u \geq &\varsigma 	\left(\Xint-_{B_{r/2}}  u(x)^tdx\right)^{1/t} -c \Big(\frac{r}{R}\Big)^\frac{s_1p}{p-1}T_{p-1}(u_-;x_0,R) -c\Big(\frac{r}{R}\Big)^\frac{s_2q}{q-1}T_{q-1}(u_-;x_0,R) \\
 		&-c r^\frac{s_1p-s_2q}{p-q}-c  \| f \|_{L^\ga(B_R)}^\frac{1}{p-1} r^\frac{\ga s_1p-N}{\ga(p-1)},
 	\end{align*}
 	for some $\varsigma\in (0,1)$.
 \end{Corollary}
 
\textbf{Proof of Theorem \ref{harnack}}:
From Corollary \ref{corlocalbd}, for all $\rho\in (0,1)$, we have
 \begin{align*}
 	\sup_{B_{\rho/2}} u \leq & \vep T_{p-1}(u_+;x_0,\frac{\rho}{2})+ \vep \rho^\frac{s_1p-s_2q}{p-1} T_{q-1}(u_+;x_0,\frac{\rho}{2})^\frac{q-1}{p-1}+ \vep \rho^\frac{s_1p-s_2q}{p-q} \nonumber\\
 	&+   \vep^\frac{p-\theta}{p}  \|f\|_{L^\ga(B_\rho)}^{1/(p-1)} + C\vep^\frac{(1-p)p^*_{s_1}}{p(p^*_{s_1}-p\sg)} \Big(\Xint-_{B_{\rho}}u_+^{p\sg}dx\Big)^{1/(p\sg)},
 \end{align*}
where  $\theta=\frac{(p^*_{s_1})'}{\ga}$ and $\sg=\frac{p-\theta}{p(1-\theta)}$. Using Lemma \ref{lem5}, the above inequality reduces to 
\begin{align*}
 \sup_{B_{\rho/2}} u \leq &C\vep^\frac{(1-p)p^*_{s_1}}{p(p^*_{s_1}-p\sg)} \Big(\Xint-_{B_{\rho}}u_+^{p\sg}dx\Big)^{1/(p\sg)} +C\vep \sup_{B_\rho} u+ C\vep \Big(\frac{\rho}{R}\Big)^\frac{s_1p}{p-1} T_{p-1}(u_-;x_0,R)
 \nonumber\\
 &+ C\vep \Big(\frac{\rho}{R}\Big)^\frac{s_2q}{p-1} \rho^\frac{s_1p-s_2q}{p-1}T_{q-1}(u_-;x_0,R)^\frac{q-1}{p-1} 
 + C\vep \rho^\frac{s_1p-s_2q}{p-1} 
 +C \vep^\frac{p-\theta}{p} \|f\|_{L^\ga(B_R)}^{1/(p-1)}.
\end{align*}
 Let $1/2\leq \kappa'<\kappa\leq 3/4$ and let $\phi\in C_c^\infty(B_{\kappa r})$ be such that $\phi\equiv 1$ in $B_{\kappa'r}$ and $|\na \phi| \leq 4/[(\kappa-\kappa')r]$. Then, by a covering argument, setting $\rho=(\kappa-\kappa')r$, we get 
 \begin{align*}
 	 \sup_{B_{\kappa'r}} u \leq &C \frac{\vep^\frac{(1-p)p^*_{s_1}}{p(p^*_{s_1}-p\sg)}}{(\kappa-\kappa')^{N/(p\sg)}} \left( \Xint-_{B_{\kappa r}}u_+^{p\sg} dx \right)^{1/(p\sg)} +C\vep \sup_{B_{\kappa r}} u+ \vep \Big(\frac{r}{R}\Big)^\frac{s_1p}{p-1} T_{p-1}(u_-;x_0,R)
 	\nonumber\\
 	&+ C\vep \Big(\frac{r}{R}\Big)^\frac{s_2q}{p-1} r^\frac{s_1p-s_2q}{p-1}T_{q-1}(u_-;x_0,R)^\frac{q-1}{p-1} 
 	+ C\vep r^\frac{s_1p-s_2q}{p-1} 
 	+C \vep^\frac{p-\theta}{p}\|f\|_{L^\ga(B_R)}^{1/(p-1)}.
 \end{align*}
 Next, upon using Young's inequality and taking $\vep=1/(4C)$, for all $t\in (0,p\sg)$, we obtain
 \begin{align*}
  C\frac{\vep^\frac{(1-p)p^*_{s_1}}{p(p^*_{s_1}-p\sg)}}{(\kappa-\kappa')^{N/(p\sg)}} \left( \Xint-_{B_{\kappa r}}u_+^{p\sg} dx \right)^{1/(p\sg)}
 	&\leq C \frac{\vep^\frac{(1-p)p^*_{s_1}}{p(p^*_{s_1}-p\sg)}}{(\kappa-\kappa')^{N/(p\sg)}}  \big( \sup_{B_{\kappa r}} u\big)^\frac{p\sg-t}{p\sg} \left( \Xint-_{B_{\kappa r}}u_+^{t} dx \right)^{1/(p\sg)}
  \\
 	&\leq \frac{1}{4} \sup_{B_{\kappa r}} u + \frac{C}{(\kappa-\kappa')^{N/t}} \left( \Xint-_{B_{\kappa r}}u_+^t dx \right)^{1/t}.
 \end{align*}
Therefore, using H\"older's inequality, for all $t\in (0,p^*_{s_1})$, we deduce that
\begin{align*}
 \sup_{B_{\kappa'r}} u &\leq  \frac{1}{2} \sup_{B_{\kappa r}} u + \frac{C}{(\kappa-\kappa')^{N/t}} \left( \Xint-_{B_{\kappa r}} u_+^t dx \right)^{1/t} +  C \sup_{B_{\kappa r}} u+ \vep \Big(\frac{r}{R}\Big)^\frac{s_1p}{p-1} T_{p-1}(u_-;x_0,R)
 	\nonumber\\
 	&\quad+ C \Big(\frac{r}{R}\Big)^\frac{s_2q}{p-1} r^\frac{s_1p-s_2q}{p-1}T_{q-1}(u_-;x_0,R)^\frac{q-1}{p-1} 
 	+ C r^\frac{s_1p-s_2q}{p-1} 
 	+C  \|f\|_{L^\ga(B_R)}^{1/(p-1)}.
\end{align*}
 Now, the rest of the proof follows using a standard iteration technique, e.g., see \cite[Proof of Theorem 1.1, p.1829]{dicastroHarn}. \QED

\section{H\"older continuity results}
In this section, we prove our main H\"older regularity results.\\  
\textbf{Proof of Theorem \ref{intreg}}: 
We first observe that 
$$u\in  X^{s_1,p}_{u}(B_{3R_0/2}(x_0),B_{2R_0}(x_0))\cap  X^{s_2,q}_{u}(B_{3R_0/2}(x_0),B_{2R_0}(x_0)).$$ 
Furthermore, $u$ solves 
\begin{equation*}
	(-\Delta)^{s_1}_{p}u+  (-\Delta)^{s_2}_{q}u = f \quad \text{in} \; B_{3R_0/2}(x_0)
\end{equation*}
and $u\in L^\infty(B_R(x_0))$ (thanks to Proposition \ref{localbdd}). For all $j\in\mb N\cup\{0\}$, define sequences $R_j=\frac{R_0}{4^j}$, $B_j=B_{R_j}$, $\frac{1}{2}B_j=B_{R_j/2}$.
We claim that there exist $\al>0$ (a generic constant), $\omega>0$, a non-decreasing sequence $\{m_j\}$ and a non-increasing sequence $\{M_j\}$ such that  
\begin{align*}
	m_j\le\inf_{B_j}\; u\le\sup_{B_j}\; u\le M_j ,\qquad M_j-m_j=\omega R_j^\al.
\end{align*}
We will proceed by induction. For $j=0$, we set $M_0= \|u\|_{L^\infty(B_{R_0})}$ and $m_0= M_0-\omega R_0^\al$, where $\omega$ satisfies 
\begin{equation}\label{eqmu}
	\omega\ge \frac{2\|u\|_{L^\infty(B_{R_0})}}{R_0^\al}>0.
\end{equation}
Hence, $m_0\le\ds\inf_{B_0}\; u\le\ds\sup_{B_0}\; u\le M_0.$ Suppose that the claim holds for all $i\in\{0,\dots, j\}$ for some $j\in\mb N$. To prove the claim for $j+1$, we first consider the case $1<\frac{N(p-1)}{N-s_1p}$. Then
\begin{align}\label{eq34}
	M_j-m_j&= \Xint-_{\frac{1}{2}B_{j+1}}(M_j-u(x))dx+\Xint-_{\frac{1}{2}B_{j+1}}(u(x)-m_j)dx \nonumber\\
	&\le \Big(\Xint-_{\frac{1}{2}B_{j+1}}(M_j-u(x))^{t}dx\Big)^{1/t}+ \Big(\Xint-_{\frac{1}{2}B_{j+1}}(u(x)-m_j)^{t}dx\Big)^{1/t},
\end{align}
for some $1<t<\frac{N(p-1)}{N-s_1p}$. Employing Corollary \ref{corweakharn}, for the choice $r=R_{j+1}$ and $R=R_j$, we obtain 
\begin{align*}
	\varsigma(M_j-m_j)&\le \inf_{B_{j+1}}(M_j-u)+  \inf_{B_{j+1}}(u-m_j) + C\sum_{(\el,s)} T_{\el-1}((M_j-u)_-;x_0,R_j) \\
	&\quad+ C\sum_{(\el,s)} T_{\el-1}((u-m_j)_-;x_0,R_j) 
	+ C R_j^\frac{ps_1-qs_2}{p-q} + C  \| f \|_{L^\ga(B_{R_0})}^\frac{1}{p-1} R_j^\frac{\ga s_1p-N}{\ga(p-1)},
\end{align*} 
that is,
\begin{align}\label{eq31}
	\underset{B_{j+1}}{\rm osc}\; u&\le \left(1-\varsigma\right)(M_j-m_j) +C\sum_{(\el,s)} T_{\el-1}((M_j-u)_-;x_0,R_j)+ C\sum_{(\el,s)} T_{\el-1}((u-m_j)_-;x_0,R_j) \nonumber\\
	&\quad + C R_j^\frac{ps_1-qs_2}{p-q}
	+C  \| f \|_{L^\ga(B_{R_0})}^\frac{1}{p-1} R_j^\frac{\ga s_1p-N}{\ga(p-1)}.
\end{align}
Now, we estimate the different tail terms appearing in the above expression:
\begin{align*}
	T_{\el-1}((u-m_j)_-;x_0,R_j)^{\el-1}= R_j^{s\el} \left[\sum_{k=0}^{j-1} \int_{B_k\setminus B_{k+1}}\frac{(u(y)-m_j)^{\el-1}}{|x_0-y|^{N+s\el}}dy +  \int_{B_0^c} \frac{(u(y)-m_j)^{\el-1}}{|x_0-y|^{N+s\el}}dy \right],
\end{align*}
where $B_0^c:=\mb R^N\setminus B_0$. By the induction hypothesis and \cite[Proof of Theorem 5.4]{iann}, for $\al\in (0,1)$, we have 
\begin{align*}
	\sum_{k=0}^{j-1} \int_{B_k\setminus B_{k+1}}\frac{(u(y)-m_j)^{\el-1}}{|x_0-y|^{N+s\el}}dy \leq c \omega^{\el-1}R_j^{\al(\el-1)-\el s}S_\el(\al),
\end{align*}
where $S_\el(\al):= \ds\sum_{k=1}^\infty \frac{(4^{\al k}-1)^{\el-1}}{4^{\el sk}}\ra 0$ as $\al\ra0^+$. For the second term, we set  $Q(u;x_0,R_0):=\|u\|_{L^\infty(B_0)}+T_{p-1}(u;x_0,R_0)+T_{q-1}(u;x_0,R_0)$. By observing that $m_j \leq \inf_{B_j} u\leq \sup_{B_j} u \leq \|u\|_{L^\infty(B_0)}$, we deduce that 
\begin{align*}
	\int_{\mb R^N\setminus B_0} \frac{(u(y)-m_j)^{\el-1}}{|x_0-y|^{N+s\el}}dy &\leq \int_{\mb R^N\setminus B_0} \frac{(|u(y)|+\|u\|_{L^\infty(B_0)})^{\el-1}}{|x_0-y|^{N+s\el}}dy \\
	&\leq C\frac{\|u\|_{L^\infty(B_0)}^{\el-1}+T_{\el-1}(u;x_0,R_0)^{\el-1}}{R_0^{sl}}.
\end{align*}
Therefore,
\begin{align*}
	T_{\el-1}((u-m_j)_-;x_0,R_j)^{\el-1}\leq c \omega^{\el-1}R_j^{\al(\el-1)}S_\el(\al) + c  R_j^{s\el} \frac{\big(\|u\|_{L^\infty(B_0)}+T_{\el-1}(u;x_0,R_0)\big)^{\el-1}}{R_0^{s\el}},
\end{align*}
and an analogous estimate holds for $T_{\el-1}((M_j-u)_-;x_0,R_j)^{\el-1}$ also. Thus,  from \eqref{eq31} and the inductive hypothesis, we get 
\begin{align*}
	\underset{B_{j+1}}{\rm osc}\; u&\le \left(1-\varsigma\right)\omega R_j^\al +c \omega R_j^{\al} S_p(\al)^{1/(p-1)} + c R_j^\frac{s_1p}{p-1} \frac{Q(u;x_0,R_0)}{R_0^{s_1p/(p-1)}} \nonumber\\ 
	&\quad+c \omega R_j^{\al}S_q(\al)^{1/(q-1)}+c R_j^\frac{s_2q}{q-1} \frac{Q(u;x_0,R_0)}{R_0^{s_2q/(q-1)}}   + C R_j^\frac{ps_1-qs_2}{p-q}+ C  \| f \|_{L^\ga(B_{R_0})}^\frac{1}{p-1} R_j^\frac{\ga s_1p-N}{\ga(p-1)}.	
\end{align*}
Then, for $\al<\min \big\{ \frac{ps_1}{p-1}, \frac{qs_2}{q-1}, \frac{ps_1-qs_2}{p-q},\frac{\ga s_1p-N}{\ga(p-1)} \big\}$, we have
\begin{align}\label{eq32}
	\underset{B_{j+1}}{\rm osc}\; u&\le 4^{\al}\left[1-\varsigma + c  S_p(\al)^\frac{1}{p-1}+ c S_q(\al)^\frac{1}{q-1} \right] \omega R_{j+1}^\al \nonumber \\
	&\quad+ \frac{c4^\al}{R_0^\al} \left[ Q(u;x_0,R_0) + R_0^\frac{ps_1-qs_2}{p-q}+  \| f \|_{L^\ga(B_{R_0})}^\frac{1}{p-1} R_0^\frac{\ga s_1p-N}{\ga(p-1)} \right] R_{j+1}^\al,
\end{align}
where in the last line we have used $R_{j}= 4^{\al}R_{j+1} \leq R_0$.
On account of $S_\el(\al)\to 0$ as $\al\to 0^+$, we can choose $\al<\min \big\{ \frac{ps_1}{p-1}, \frac{qs_2}{q-1}, \frac{ps_1-qs_2}{p-q},\frac{\ga s_1p-N}{\ga(p-1)} \big\}$ small enough such that 
\[ 4^\al\left(1-\varsigma+ c  S_p(\al)^\frac{1}{p-1}+ c S_q(\al)^\frac{1}{q-1} \right)\le 1-\frac{\varsigma}{4}\] 
and set
\begin{equation}\label{eqmu1}
	\omega=\frac{4^{\al+1}\; c}{\varsigma\;R_0^\al}\Big( Q(u;x_0,R_0)+R_0^\frac{ps_1-qs_2}{p-q}+ \| f \|_{L^\ga(B_{R_0})}^\frac{1}{p-1} R_0^\frac{\ga s_1p-N}{\ga(p-1)} \Big).
\end{equation} 
Note that with the above choice of $\omega$, \eqref{eqmu} is satisfied if the constant $c$, appearing in \eqref{eq32}, is replaced by a bigger constant such that $4c/\varsigma\geq 2$.  Thus, from \eqref{eq32}, we have
\begin{align*}
	\underset{B_{j+1}}{\rm osc}\; u\le \omega R_{j+1}^\al.
\end{align*}
Therefore, we pick $m_{j+1},\; M_{j+1}$ such that 
\begin{align*}
	m_j\le m_{j+1}\le\inf_{B_j}\; u\le\sup_{B_j}\; u\le M_{j+1}\le M_j \quad \mbox{and } M_{j+1}-m_{j+1}= \omega R_{j+1}^\al.
\end{align*}
To finish the proof of the theorem, we fix $r\in(0,R_0)$. Let $j\in\mb N\cup\{0\}$ be such that $R_{j+1}< r\le R_j$, then taking into account \eqref{eqmu1} and $R_j\le 4r$, we have 
\begin{align*}
	\underset{B_r}{\rm osc}\; u\le\underset{B_j}{\rm osc} \;u\le\omega R_j^\al \le C\Big(Q(u;x_0,R_0)+R_0^\frac{ps_1-qs_2}{p-q}+ \| f \|_{L^\ga(B_{R_0})}^\frac{1}{p-1} R_0^\frac{\ga s_1p-N}{\ga(p-1)} \Big)\frac{r^\al}{R_0^\al}.
\end{align*}
For the case $\frac{N(p-1)}{N-s_1p}\leq 1$ (this forces $p<2$), using the induction hypothesis, we observe that  
\begin{align*}
	M_j-u \leq (M_j-m_j)^{2-p} (M_j-u)^{p-1}, \quad\mbox{in }B_j, 
\end{align*}
and this still holds for $(u-m_j)$. Hence,
\begin{align*}
	M_j-m_j&\leq (M_j-m_j)^{2-p} \Big[
	\Xint-_{\frac{1}{2}B_{j+1}}(M_j-u(x))^{p-1} dx+ \Xint-_{\frac{1}{2}B_{j+1}}(u(x)-m_j)^{p-1}dx \Big].
\end{align*}
This implies that
\begin{align*}
	M_j-m_j&\leq  \Big[
	\Xint-_{\frac{1}{2}B_{j+1}}(M_j-u(x))^{p-1}dx+ \Xint-_{\frac{1}{2}B_{j+1}}(u(x)-m_j)^{p-1}dx \Big]^{1/(p-1)}\\
	&\leq 2^\frac{2-p}{p-1} \left[
	\Big(\Xint-_{\frac{1}{2}B_{j+1}}(M_j-u)^{p-1}dx\Big)^{1/(p-1)}+ \Big(\Xint-_{\frac{1}{2}B_{j+1}}(u-m_j)^{p-1}dx\Big)^{1/(p-1)} \right],
\end{align*}
which is similar to \eqref{eq34}, with $t=p-1$. Multiplying both the sides of the above expression by $\bar{\varsigma}:=\varsigma/2^\frac{2-p}{p-1}$, we can employ Corollary \ref{corweakharn}. Then, the rest of the proof follows similarly as before with $\bar{\varsigma}$ in place of $\varsigma$. 
This completes the proof of the theorem.\QED

\textbf{Proof of Corollary \ref{bdryreg}}: The proof follows by using the boundary behavior of the solution, given by \cite[Proposition 3.11]{JDS2}, and the interior regularity result of Theorem \ref{intreg}. For details, see the proof of \cite[Theorem 1.1]{iann}. \QED

\section{Strong maximum principle}
The main goal of this section is to prove our strong maximum principle. We first recall the notion of viscosity solution (see e.g., \cite{korvenpaa2}). To this end, for $u:\mb R^N\to\mb R$ and $D\subset\Om$, set the following: 
\begin{align*}
	&N_u:=\{x\in\Om : \ \na u(x)=0\}, \quad d_u(x):=\mathrm{dist}(x,N_u) \quad\mbox{and}\\
	&C^2_\ba(D):=\biggl\{ u\in C^2(D): \ \sup_{x\in D} \Big( \frac{\min\{d_u(x),1\}^{\ba-1}}{|\na u(x)|}+\frac{|D^2u(x)|}{d_u(x)^{\ba-2}} \Big)<\infty \biggr\}.
\end{align*}
\begin{Definition}\label{defnvsc}
	A function $u:\mb R^N\to [-\infty,\infty]$ is said to be a viscosity  super-solution to $(-\De)_p^{s_1} u+ (-\De)_q^{s_2} u\geq 0$ in $\Om$, if the following hold:
	\begin{itemize}
		\item[(i)] $u$ is lower semi-continuous in $\Om$ such that $u <\infty$ a.e. in $\mb R^N$ and $u>-\infty$ everywhere in $\Om$;
		\item[(ii)] $u_-\in L_{s_1p}^{p-1}(\mb R^N)\cap L_{s_2q}^{q-1}(\mb R^N)$;
		\item[(iii)] if whenever $B_r(x_0)\subset\Om$ and $\phi\in C^2(B_r(x_0))$ are such that 
		\begin{align*}
			\phi(x_0)=u(x_0) \quad \mbox{and } u(x) \geq \phi(x) \ \ \mbox{in } B_r(x_0),
		\end{align*}
		and one of the following holds
		\begin{itemize}
			\item[(a)] $p>\frac{2}{2-s_1}$ and $q>\frac{2}{2-s_2}$ or $\na\phi(x_0)\neq 0$,
			\item[(b)] either  $p\leq\frac{2}{2-s_1}$ or $q\leq\frac{2}{2-s_2}$; $\na \phi(x_0)=0$ such that $x_0$ is an isolated critical point of $\phi$ and $\phi\in C^2_{\ba}(B_r(x_0))$, for some $\ba>\max\{\frac{s_1p}{p-1},\frac{s_2q}{q-1}\}$;
		\end{itemize}
		then $(-\De)_{p}^{s_1} \phi_r(x_0) +(-\De)_q^{s_2} \phi_r(x_0) \geq  0$, where 
		\begin{align}\label{eqA15}
			\phi_r(x)=\begin{cases}
				\phi(x) &\mbox{if }x\in B_r(x_0),\\
				u(x) &\mbox{if }x\in \mb R^N\setminus B_r(x_0).
			\end{cases}
		\end{align}
	\end{itemize}
\end{Definition}
We recall the following weak comparison principle (e.g., see \cite{DDS}).
\begin{Proposition}
 Let $\widetilde{W}^{s,p}(\Om):= \{u\in L^p_{\rm loc}(\mb R^N)\cap L^{p-1}_{sp}(\mb R^N): \exists\;U\Supset\Om \ \mbox{with }\ \|u\|_{W^{s,p}(U)} <\infty \}$.
 Assume that $u,v\in \widetilde{W}^{s_1,p}(\Om)$ are such that 
  {\small\begin{align*}
  	\sum_{(\el,s)}{\int_{\mb R^{N}}\int_{\mb R^N}} \frac{[u(x)-u(y)]^{\el-1}(\phi(x)-\phi(y))}{|x-y|^{N+s\el}}dxdy \leq \sum_{(\el,s)}{\int_{\mb R^{N}}\int_{\mb R^N}} \frac{[v(x)-v(y)]^{\el-1}(\phi(x)-\phi(y))}{|x-y|^{N+s\el}}dxdy
  \end{align*}}
 for all $\phi\in W^{s_1,p}_0(\Om)$ and $u\leq v$ in $\Om^c$. Then, $u\leq v$ in $\Om$.
\end{Proposition}
Now, we prove that continuous weak super-solutions are viscosity super-solutions.
\begin{Lemma}\label{lemA1}
	Let $u\in W^{s_1,p}_0(\Om)\cap C(\ov\Om)$ be a weak super-solution to $(-\De)_p^{s_1} u+ (-\De)_q^{s_2} u\geq 0$ in $\Om$. Then, $u$ is also a viscosity super-solution.
\end{Lemma} 
\begin{proof}
	From the assumption on $u$, it is clear that items (i) and (ii) of Definition \ref{defnvsc} are satisfied. To prove (iii), on the contrary, assume that there exist $x_0\in\Om$ and $\phi\in C^2(B_r(x_0))$ such that $\phi(x_0)= u(x_0)$, $u(x)\geq \phi(x)$ in $B_r(x_0)$, either (a) or (b) of Definition \ref{defnvsc}(iii) holds and 
	\[ (-\De)_{p}^{s_1} \phi_r(x_0) +(-\De)_q^{s_2} \phi_r(x_0) <0, \]
	for some $r>0$ and $\phi_r$ given by \eqref{eqA15}. By \cite[Lemma 3.8]{korvenpaa2}, we have  $(-\De)_{p}^{s_1} \phi_r(x)$ and $(-\De)_q^{s_2} \phi_r(x)$ are continuous at $x_0$. Therefore, there exist $\rho'\in (0,r)$ and $\eta>0$ such that 
	\begin{align*}
		(-\De)_{p}^{s_1} \phi_r(x) +(-\De)_q^{s_2} \phi_r(x) <-\eta \quad\mbox{for all }x\in B_{\rho'}(x_0).
	\end{align*}
 Following the proof of \cite[Lemma 3.9]{korvenpaa2}, there exist $\e>0$, $\rho\in (0,\rho'/2)$ and $b\in C_c^2(B_{\rho/2}(x_0))$ such that $b(x_0)=1$ with $0\leq b \leq 1$, and $\psi_\e(x)=\phi_r(x)+\e b(x)$ satisfies 
		\begin{align*}
			\sup_{B_\rho(x_0)}|(-\De)_p^{s_1} \psi_\e(x)-(-\De)_p^{s_1} \phi_r(x)| <\frac{\eta}{2} \quad\mbox{and } \sup_{B_\rho(x_0)}|(-\De)_q^{s_2} \psi_\e(x)-(-\De)_q^{s_2} \phi_r(x)|<\frac{\eta}{2}.
	\end{align*}
	Consequently, we obtain 
	\[ (-\De)_{p}^{s_1} \psi_\e(x)+(-\De)_q^{s_2} \psi_\e(x) \leq 0 \quad\mbox{for all }x\in B_\rho(x_0). \]
	Let $v\in W^{s_1,p}_0(B_\rho(x_0))$ be a non-negative function, then multiplying the above equation with it and upon integration, we get
	{\small\begin{align*}
		\sum_{(\el,s)}{\int_{\mb R^{N}}\int_{\mb R^N}} \frac{[\psi_\e(x)-\psi_\e(y)]^{\el-1}(v(x)-v(y))}{|x-y|^{N+s\el}}d\mu \leq 0 \leq
		\sum_{(\el,s)}{\int_{\mb R^{N}}\int_{\mb R^N}} \frac{[u(x)-u(y)]^{\el-1}(v(x)-v(y))}{|x-y|^{N+s\el}}d\mu,
	\end{align*}}
 where $d\mu=dxdy$. Further, $\psi_\e\leq u$ in $B_\rho(x_0)^c$. Therefore, by the weak comparison principle, we obtain $\psi_\e \leq u$ in $B_\rho(x_0)$. But this contradicts $\psi_\e(x_0)=\phi(x_0)+\e b(x_0)=u(x_0)+\e>u(x_0)$. Thus, $u$ is a viscosity super-solution.  This completes the proof of the lemma.\QED
\end{proof}

\begin{Lemma}
 Suppose $1<q\leq p<\infty$. Let $u\in W^{s_1,p}_0(\Om)\cap C(\ov\Om)$ be a weak super-solution to $(-\De)_p^{s_1} u+ (-\De)_q^{s_2} u\geq 0$ in $\Om$ and $u\geq 0$ a.e. in $\mb R^N\setminus\Om$. Then, either $u\equiv 0$ a.e. in $\mb R^N$ or $u>0$ in $\Om$.
\end{Lemma}
\begin{proof} 
We will show that if $u\not\equiv 0$, then $u>0$ in $\Om$. By the weak comparison principle, we have $u\geq 0$ a.e. in $\mb R^N$. \\
\textit{Case I}: If $\Om$ is connected.\\
We proceed similarly to \cite[Theorem A.1]{brasco3}.
 Let $K\Subset\Om$ be any connected compact such that $u\not\equiv 0$ in $K$. Then, we will show that $u>0$ a.e. in $K$.
Since $K$ is compact, $K\subset \{x\in\Om: \ {\rm dist}(x,\pa\Om)>2r\}$, for some $r>0$. Moreover, there exists a finite covering $\{B_{r/2}(x_i)\}_{\{i=1,2\dots,n\}}$ for $K$ such that 
\begin{align}\label{eqA16}
	|B_{r/2}(x_i)\cap B_{r/2}(x_{i+1})|>0 \quad\mbox{for all } i=1,2,\dots,n-1.
\end{align}
Suppose $u\equiv 0$ on some subset of $K$ with positive measure. Then, for some $i\in\{1,\dots, n-1\}$, 
\begin{align*}
	|E:=\{x\in B_{r/2}(x_i)  : \ u(x)=0 \}|>0.
\end{align*}
For $\la>0$, set 
\begin{align*}
	U_\la(x):=\log\Big(1+\frac{u(x)}{\la}\Big), \quad \mbox{for all } x\in B_{r/2}(x_i).
\end{align*}
By observing that $U_\la\equiv 0$ on $E$, for $x\in B_{r/2}(x_i)$ and $y\in E$, we have 
\begin{align*}
	|U_\la(x)|^q =\frac{|U_\la(x)-U_\la(y)|^q}{|x-y|^{N+qs_2}}|x-y|^{N+qs_2},
\end{align*}
which upon integration with respect to $y\in E$ and $x\in B_{r/2}(x_i)$ yields
\begin{align}\label{eqA17}
	|E|\int_{B_{r/2}(x_i)} |U_\la(x)|^q dx \leq c r^{N+qs_2} \int_{B_{r/2}(x_i)}\int_{B_{r/2}(x_i)} \frac{|U_\la(x)-U_\la(y)|^q}{|x-y|^{N+qs_2}} dxdy.
\end{align}
On account of 
\begin{align*}
	\bigg|\log\frac{u(x)+\la}{u(y)+\la}\bigg|^q = |U_\la(x)-U_\la(y)|^q
\end{align*}
and \eqref{eqA18}, we deduce from \eqref{eqA17} that 
\begin{align*}
	\int_{B_{r/2}(x_i)} \bigg| \log\Big(1+\frac{u(x)}{\la}\Big) \bigg|^q dx \leq c\frac{r^{N+s_2q}}{|E|} \big( r^{N-s_2q}+r^{N-s_1p}(\|u\|_{L^\infty(B_R)}+\la)^{p-q} \big),
\end{align*}
where we have used that $u\geq 0$ in $\mb R^N$ (consequently, $u_-=0$). Passing to the limit as $\la\to 0$ in the above expression, we get 
\begin{align*}
	u=0 \quad \mbox{a.e. in }B_{r/2}(x_i).
\end{align*}
By using the property \eqref{eqA16}, we can proceed similarly for balls $B_{r/2}(x_{i-1})$ and $B_{r/2}(x_{i+1})$ (note that $|E:=\{x\in B_{r/2}(x_{i+1})  : \ u(x)=0 \}|\geq |B_{r/2}(x_{i+1})\cap B_{r/2}(x_{i})|>0$) and so on for all $i\in\{1,\dots,n\}$, that is, $u=0$ a.e. on $K$. This is a contradiction to our assumption that $u\not\equiv 0$ in $K$. Thus $u>0$ a.e. in $K$.\\
Since $\Om$ is open and connected (path connected), there exists a sequence of compact connected sets $\Om_n\subset\Om$ such that 
	\begin{align*}
		\Om=\cup_{n\in\mb N}\Om_n, \quad \Om_n\subset\Om_{n+1} \quad\forall \ n\in\mb N\quad\mbox{and }u\not\equiv 0 \quad\mbox{in }\Om_n,\quad\mbox{for all }n\geq n_0.
	\end{align*}
	Then, proceeding as above, we get $u>0$ a.e. in $\Om_n$ for all $n\geq n_0$. Therefore, $u>0$ a.e. in $\Om$.  \\
\textit{Case II}: Let $\Om$ be any bounded domain.\\
By the weak comparison principle we have $u\geq 0$ a.e. in $\mb R^N$. Next, we will show that if $u\not\equiv 0$ in $\Om$, then $u\not\equiv 0$ in every connected component of $\Om$. Suppose there exists a connected component $E$ of $\Om$ such that $u\equiv 0$ in $E$. Let $\psi\in W^{s_1,p}_0(E)$ be a test function. Then, 
\begin{align*}
	0 \leq \sum_{(\el,s)} \int_{\mb R^N} \int_{\mb R^N}  \frac{[u(x)- u(y)]^{\el-1}}{|x-y|^{N+s\el}} \big(\psi(x)-\psi(y)\big)dxdy= -2 \sum_{(\el,s)}\int_{E} \int_{E^c}  \frac{\big( u(y)\big)^{\el-1} \psi(x)}{|x-y|^{N+s\el}}dxdy.
\end{align*}
Thus, $u=0$ in $E^c$, that is, $u=0$ a.e. in $\mb R^N$. This is a contradiction to the assumption $u\not\equiv 0$. \\
To complete the proof of the lemma, we will show that if there exists $x_0\in\Om$ such that $u(x_0)=0$, then $u(x)=0$ a.e. $x\in\mb R^N$. By the above discussion, we have either $u>0$ a.e. in $\Om$ or $u=0$ a.e. in $\mb R^N$. 
By Lemma \ref{lemA1}, we have that $u$ is a viscosity super-solution. Thus, proceeding as in \cite[Lemma 3.5]{delpezzo}, we get $u>0$ for all $x\in\Om$. \QED
\end{proof}

Proceeding similarly to the proof of \cite[Theorem 3.12]{JDS2}, we have the following result.
\begin{Proposition}\label{sub}
	Let $1<q\leq p<\infty$. Then, for every $\vartheta>0$, there exists a unique solution $w_{\vartheta}\in W^{s_1,p}_0(\Om)\cap C^{0,\al}(\ov\Om)$, for some $\al\in (0,s_1)$, of the following problem:
	\begin{equation*}
		\left\{\begin{array}{rllll}	  
			(-\Delta)^{s_1}_{p}u+  (-\Delta)^{s_2}_{q}u &= \vartheta, \; \; u>0 \quad \mbox{in } \Om, \\
			u &=0 \quad \mbox{in }  \mb R^N\setminus \Om.
		\end{array}
		\right.\tag{$Q_\vartheta$}\label{probsub}
	\end{equation*}
	Moreover, $w_\vartheta\to 0$ in $C^{0,\sg}(\ov\Om)$, as $\vartheta\to 0$, for all $\sg<\al$.
\end{Proposition}
\textbf{Proof of Theorem \ref{strngmax}}:
	Without loss of generality we may assume that $g$ is non-decreasing and $g(0)=0$ (by Jordan's decomposition). Since $u\not\equiv 0$, there exist $x_0\in\Om$, $\rho,\e>0$ and $\vth_0\in (0,1)$ such that
	\begin{align}\label{eqA11}
		\sup_{\ov{B_\rho(x_0)}} w_{\vth_0} \le \inf_{\ov{B_\rho(x_0)}} u -\e,
	\end{align}
	where $w_{\vth_0}$ is the solution to problem $(Q_{\vth_0})$. Indeed, there exists $x_0\in \Om$ such that $u(x_0)>0$. Then by continuity of $u$ and Proposition \ref{sub} (in particular, $w_\vth\to 0$ in $C_0(\ov\Om)$) as $\vth\to 0$, for $\e<u(x_0)/4$, there exist $\rho>0$ and $\vth_0\in (0,1)$  such that 
	\begin{align*}
		\|w_{\vth_0}\|_{C(\ov\Om)} \leq u(x_0)/2\leq u(x_0)-2\e < u(x)-\e \quad\mbox{for all }x\in B_\rho(x_0).
	\end{align*}
	For all $\vth\in (0,\vth_0]$, set the following:
	\begin{align*}
		v_\vth:=\begin{cases}
			w_\vth &\mbox{in }\mb R^N\setminus\ov{B_{\rho/2}(x_0)}, \\
			u &\mbox{in }\ov{B_{\rho/2}(x_0)},
		\end{cases}
	\end{align*}
	where $w_{\vth}$ is the solution to problem \eqref{probsub}. Since $w_\vth\leq w_{\vth_0}$, on account of \eqref{eqA11}, we have $v_\vth\leq u$ in $\ov{B_{\rho}(x_0)}$ and $v_\vth\in \widetilde{W}^{s_1,p}(\Om\setminus\ov{B_{\rho}(x_0)})$. By the nonlocal superposition principle \cite[Proposition 2.6]{iann} and proceeding as in \cite[Theorem 2.6]{iannmospap}, we have weakly in $\Om\setminus\ov{B_{\rho}(x_0)}$
	\begin{align*}
		(-\De)_p^{s_1}v_\vth  \leq (-\De)_p^{s_1}w_\vth - C_\rho \e^{p-1}  \quad\mbox{and }(-\De)_q^{s_2}v_\vth \leq  (-\De)_q^{s_2}w_\vth - C'_\rho \e^{q-1}.
	\end{align*}
	Choosing $\vth\in (0,\vth_0]$ small enough, for a positive constant $C$ (independent of $\e$), we obtain
	\begin{align*}
		(-\De)_p^{s_1}v_\vth + (-\De)_q^{s_2}v_\vth \leq \vth-c_\rho \e^{p-1} \leq -C\e^{p-1}.
	\end{align*}
	Moreover, on account of the fact that $g(v_\vth)\to $ uniformly in $\Om\setminus\ov{B_{\rho}(x_0)}$, as $\vth\to 0$, for even smaller $\vth$ (if necessary), we get 
	\begin{align*}
		(-\De)_p^{s_1}v_\vth + (-\De)_q^{s_2}v_\vth +g(v_\vth)\leq 0 \leq (-\De)_p^{s_1}u + (-\De)_q^{s_2}u +g(u) \quad\mbox{ weakly in } \Om\setminus\ov{B_{\rho}(x_0)}
	\end{align*}
	and $v_\vth\leq u$ in $\mb R^N\setminus(\Om\setminus\ov{B_{\rho}(x_0)})$. Thus, by the weak comparison, we obtain $v_\vth\leq u$ in $\Om\setminus\ov{B_{\rho}(x_0)}$. Consequently, using \eqref{eqA11}, we have
	\begin{align*}
		u \geq w_\vth \quad\mbox{in }\Om.
	\end{align*} 
	Since $0<w_\vth\in C(\ov\Om)$, it is evident from the proof of \cite[Proposition 2.6]{JDS2} that $w_\vth\geq c_1 d^{s_1}$, for some positive constant $c_1$. Hence, the result of the theorem follows. \QED

\end{document}